\documentclass{amsart}
\usepackage{amsmath,amsthm,amssymb}
\usepackage{graphicx}

\newtheorem{theorem}{Theorem}[section]
\newtheorem{proposition}[theorem]{Proposition}
\newtheorem{lemma}[theorem]{Lemma}
\newtheorem{corollary}[theorem]{Corollary}

\theoremstyle{definition}
\newtheorem{definition}[theorem]{Definition}
\newtheorem{remark}[theorem]{Remark}

\title[Stein 4-manifolds and corks]{Stein 4-manifolds and corks}
 
\author[AKBULUT and YASUI]{Selman Akbulut and Kouichi Yasui}
\thanks{The first named author is partially supported by NSF grant DMS 0905917, and the second named author was partially supported by GCOE, Kyoto University, by KAKENHI~21840033 and by JSPS Research Fellowships for Young Scientists.}
\date{November 8, 2010.\: \textit{Revised}: September 6, 2012.}
\address{Department~of~Mathematics, Michigan State University, E.Lansing, MI, 48824, USA}
\email{akbulut@math.msu.edu}

\address{Department~of~Mathematics, Graduate School~of~Science, Hiroshima~University, 1-3-1 Kagamiyama, Higashi-Hiroshima, Hiroshima, 739-8526, Japan}\email{kyasui@hiroshima-u.ac.jp}

\begin{document}
%\maketitle

\begin{abstract}

It is known that every compact Stein $4$-manifolds can be embedded into a simply connected, minimal, closed, symplectic $4$-manifold. By using this property, we discuss a new method of constructing corks. This method generates a large class of new corks including all the previously known ones. We prove that every one of these corks can knot infinitely many different ways in a closed smooth manifold, by showing that cork twisting along them gives different exotic smooth structures. We also give an example of infinitely many disjoint embeddings of a fixed cork into a non-compact $4$-manifold which produce infinitely many exotic smooth structures. Furthermore, we construct arbitrary many simply connected compact codimension zero submanifolds of $S^4$ which are mutually homeomorphic but not diffeomorphic.

%It is known that every exotic smooth structure on a simply connected closed $4$-manifold is determined by a codimention zero compact contractible Stein submanifold and an involution on its boundary. Such a pair is called a cork. In this paper, we construct infinitely many knotted imbeddings of corks in $4$-manifolds such that they induce infinitely many different exotic smooth structures. We also show that we can imbed an arbitrary finite number of corks disjointly into 4-manifolds, so that the corresponding involutions on the boundary of the contractible $4$-manifolds give mutually different exotic structures. Furthermore, we construct similar examples for plugs. 
\end{abstract}

%\keywords{4-manifold; handlebody; Stein manifold; rational blow-down; $h$-cobordism}
\maketitle

\vspace{-.1in}

\section{Introduction}It is known that every smooth structure on a simply connected closed smooth $4$-manifold is obtained from the given manifold by a cork twist (Matveyev~\cite{M}, Curtis-Freedman-Hsiang-Stong~\cite{C}, Akbulut-Matveyev~\cite{AM2}). It is thus important to investigate cork structures of $4$-manifolds. The first author~\cite{A1}, \cite{A5}, Bi\v zaca-Gompf~\cite{BG}, and the authors~\cite{AY1} have found cork structures of (surgered) elliptic surfaces. However, finding cork structures of given exotic pairs of smooth $4$-manifolds is usually a quite difficulut task. 

\vspace{.05in} 

In~\cite{AY3} (and also \cite{AY1}), rather than trying to locate corks in exotic manifold pairs we constructed exotic manifold pairs from given corks. Our strategy in \cite{AY3} was: $(1)$ Construct a ``suitable'' $4$-manifold with boundary which contains candidates of corks; $(2)$ Embed it into a closed $4$-manifold with the non-vanishing Seiberg-Witten invariant, ``appropriately''; $(3)$ Do surgeries and compute Seiberg-Witten invariants; $(4)$ Relate surgeries and cork twists. Nevertheless, the step $(2)$ is generally not easy. 

\vspace{.05in} 

Eliashberg~\cite{E1} proved that compact Stein $4$-manifolds can be recognized by handlebody pictures, that is, just by checking Thurston-Bennequin framings of its 2-handles (e.g. Gompf-Stipsicz~\cite{GS}, Ozbagci-Stipsicz~\cite{OS1}). Furthermore, compact Stein $4$-manifolds satisfy the following useful embedding theorem (though ``simply connected'' is not claimed in~\cite{LM1} and~\cite{AO3}, this is obvious from the proof in~\cite{AO3}):

\vspace{.05in} 

\begin{theorem}[Lisca-Mati\'{c}~\cite{LM1}, Akbulut-Ozbagci~\cite{AO3}]\label{th:closing of Stein}
Every compact Stein $4$-manifold with boundary can be embedded into a simply connected, minimal, closed, symplectic $4$-manifold with $b_2^+>1$. Here minimal means that there are no smoothly embedded $2$-sphere with the self-intersection number $-1$. 
\end{theorem}

In this paper, by using this embedding theorem, we give simple constructions of the various cork structures found in \cite{AY3}. Moreover, we prove some of these structurs for any cork of Mazur type, unlike the previous construction. We also construct a new example in the non-compact case. Our new strategy is as follows (the non-closed case is different): $(1)$ Construct a ``suitable'' compact Stein $4$-manifold  which contains candidates of corks (possibly after blow ups); $(2)$ Embed it into a minimal closed symplectic $4$-manifold; $(3)$ Do surgeries (possibly after blow ups) and compute Seiberg-Witten invariants; $(4)$ Relate surgeries and cork twists. Now, by this approach, the previously difficult step  $(2)$ is automatically achieved. However, in this case, we need a care on computations of Seiberg-Witten invariants, since we do not know the basic classes of the closed symplectic 4-manifold. 

\vspace{.05in} 

It is a natural question whether every smooth structure on a $4$-manifold can be induced from a fixed cork $(C,\tau)$.
The following theorem (see Section~\ref{sec:knotted cork}) shows that infinitely many different smooth structures on a closed $4$-manifold can be obtained from a fixed cork. Though such an example was given in \cite{AY3} for specific corks, the new method works for any cork of Mazur type. Similarly to \cite{AY3}, we also use Fintushel-Stern's knot surgery for the construction. 
\vspace{.05in} 

\begin{theorem}\label{th:knotting corks}Let $(C, \tau)$ be any cork of Mazur type. Then 
there exist infinitely many simply connected closed smooth $4$-manifolds $X_n$ $(n\geq 0)$ with the following properties:\smallskip 

\begin{itemize}

\item[(1)]  $X_n$ $(n\geq 0)$ are mutually homeomorphic but not diffeomorphic;

\item[(2)] For each $n\geq 1$, $X_n$ is obtained from $X_0$ by a cork twist along $(C, \tau)$. Consequently, the pair $(C,\tau)$ is a cork of $X_0$. 
 
\end{itemize}

\vspace{.05in}
In particular, from $X_0$ we can produce infinitely many different smooth structures by the cork twist along $(C, \tau)$. Consequently, these embeddings of $C$ into $X_0$ are mutually non-isotopic $($knotted copies of each other$)$. 
\end{theorem}

The next theorem (see Section~\ref{sec:construction}) says that we can put finitely many corks into mutually disjoint positions in closed $4$-manifolds so that corresponding cork twists produce mutually different exotic smooth structures on the $4$-manifolds. We prove this by using Fintushel-Stern's rational blowdown. 

\begin{theorem}[\cite{AY3}]\label{th:disjoint corks}
For each $n\geq 1$, there exist simply connected closed smooth $4$-manifolds $X_{i}$ $(0\leq i\leq n)$ and corks $(C_i,\tau_i)$ $(0\leq i\leq n)$ of $X_0$ with the following properties:
%codimension zero compact contractible Stein submanifolds $C_i$ $(1\leq i\leq n)$ of $Y_0$, and an involution $\tau_i$ on the each boundary  $\partial C_i$ $(1\leq i\leq n)$ with the following properties:\smallskip 

\begin{itemize}

\item[(1)] The submanifolds $C_i$ $(1\leq i\leq n)$ of $X_{0}$ are mutually disjoint;\smallskip

\item[(2)] $X_{i}$ $(1\leq i\leq n)$ is obtained from $X_{0}$ by the cork twist along $(C_i,\tau_i)$;\smallskip 

\item[(3)] $X_{i}$ $(0\leq i\leq n)$ are mutually homeomorphic but not diffeomorphic. 

\end{itemize}

\end{theorem}

\vspace{.05in} 

This theorem easily gives the following corollary which says that, for an embedding of a cork into a closed $4$-manifold, cork twists can produce finitely many different exotic smooth structures on the $4$-manifold by only changing the involution of the cork without changing its embedding: 

\begin{corollary}[\cite{AY3}]\label{cor:involutions of corks}
For each $n\geq 1$, there exist simply connected closed smooth $4$-manifolds $X_{i}$ $(0\leq i\leq n)$, an embedding of a compact contractible Stein $4$-manifold $C$ into $X_{0}$, and involutions $\tau_i$ $(1\leq i\leq n)$ on the boundary $\partial C$ with the following properties: 

\begin{itemize}

\item[(1)] For each $1\leq i\leq n$, $X_{i}$ is obtained from $X_{0}$ by the cork twist along $(C,\tau_i)$ where the above embedding of $C$ is fixed;\smallskip 

\item[(2)] $X_{i}$ $(0\leq i\leq n)$ are mutually homeomorphic but not diffeomorphic, hence the  pairs $(C, \tau_i)$ $(1\leq i\leq n)$ are mutually different corks of $X_{0}$.

\end{itemize}
\end{corollary}

In \cite{AY1}, we introduced new objects which we called {\it plugs}. We can also give simple constructions of various examples for plug structures. 

\vspace{.05in} 

In this paper we also give a new example of cork structures in the non-compact case by using the above embedding theorem. The theorem below (see Section~\ref{sec:Infinitely many disjointly knotted cork}) shows that we can embed a fixed cork into infinitely many mutually disjoint positions in a non-compact $4$-manifold so that corresponding cork twists produce infinitely many mutually different exotic smooth structures on the $4$-manifold:

\begin{theorem}\label{th:infinite corks}Let $(C, \tau)$ be any cork of Mazur type. Then there exist infinitely many simply connected non-compact smooth $4$-manifolds $X_{n}$ $(n\geq 0)$ and infinitely many embedded copies $C_n$ $(n\geq 1)$ of $C$ into $X_0$ with the following properties:\smallskip 
%codimension zero compact contractible Stein submanifolds $C_i$ $(1\leq i\leq n)$ of $Y_0$, and an involution $\tau_i$ on the each boundary  $\partial C_i$ $(1\leq i\leq n)$ with the following properties:\smallskip 

\begin{itemize}

\item[(1)]  $C_n$ $(n\geq 0)$ are mutually disjoint in $X_0$;\smallskip

\item[(2)] $X_{n}$ $(n\geq 1)$ is obtained from $X_{0}$ by the cork twist along $(C_n,\tau)$;\smallskip 

\item[(3)] $X_{n}$ $(n\geq 0)$ are mutually homeomorphic but not diffeomorphic. 

\end{itemize}

Consequently, these infinitely many disjoint embeddings of $C$ into $X_0$ are mutually non-isotopic $($knotted copies of each other$)$. 
%
%$(4)$ The $4$-manifold $X_{i}$ $(i\geq 0)$ can be embedded into $S^4$. 
\end{theorem}

\vspace{.05in} 

The proof of this theorem immediately give the following corollary which says that the smallest 4-manifold $S^4$ has arbitrary many compact submanifolds which are mutually homeomorphic but not diffeomorphic. Furthermore they are obtained by a disjointly embedded fixed cork. 

\vspace{.05in}  

\begin{corollary}\label{cor:finite exotic in S^4}Let $(C, \tau)$ be any cork of Mazur type. Then, for each $n\geq 1$, there exist simply connected compact smooth $4$-manifolds $X^{(n)}_{i}$ $(0\leq i \leq n)$ and embedded copies $C_i$ $(1\leq i\leq n)$ of $C$ into $X_0$ with the following properties:\smallskip 

\begin{itemize}
\item[(1)] $C_i$ $(1\leq i\leq n)$ are mutually disjoint in $X_0$; 

\item[(2)] $X^{(n)}_{i}$ $(1\leq i \leq n)$ is obtained from $X_0$ by the cork twist along $(C_i,\tau)$;

\item[(3)] $X^{(n)}_{i}$ $(0\leq i \leq n)$ can be embedded into $S^4$;  

\item[(4)] $X^{(n)}_{i}$ $(0\leq i \leq n)$ are mutually homeomorphic but not diffeomorphic.

\end{itemize}
%\smallskip 
%
%$(3)$ The homology groups of $X^{(n)}_{i}$ $(0\leq i \leq n)$ are isomorphic to those of the boundary sum $\natural_n(S^2\times D^2)$. 
\end{corollary}\vspace{.05in} 
After the first draft of this paper,  a more systematic construction of exotic (Stein) 4-manifolds was given in our paper~\cite{AY5}. 
%The plan of this paper is as follows: In Section~\ref{sec:cork}, In Section~\ref{sec:knotted cork}, we prove Theorem~\ref{th:knotting corks}
\medskip \\
\textbf{Acknowledgements.} The second author would like to thank Kenji Fukaya and Yuichi Yamada for useful comments.

\section{Corks}\label{sec:cork}
In this section, we recall corks. For details, see~\cite{AY1}.

\begin{definition}\label{sec:cork:def:cork}
Let $C$ be a compact contractible Stein $4$-manifold with boundary and $\tau: \partial C\to \partial C$ an involution on the boundary. 
We call $(C, \tau)$ a \textit{Cork} if $\tau$ extends to a self-homeomorphism of $C$, but cannot extend to any self-diffeomorphism of $C$. 
For a cork $(C,\tau)$ and a smooth $4$-manifold $X$ which contains $C$, a {\it cork twist of $X$ along $(C,\tau)$ }is defined to be the smooth $4$-manifold obtained from $X$ by removing the submanifold $C$ and regluing it via the involution $\tau$.  Note that, a cork twist does not change the homeomorphism type of $X$ (see the remark below). A cork $(C, \tau)$ is called a cork of $X$ if the cork twist of $X$ along $(C, \tau)$ is not diffeomorphic to $X$. 
%A cork $(C, \tau)$ is called a cork of a smooth $4$-manifold $X$, if  $C\subset X$ and $X$ changes its diffeomorphism type when we remove $C$ and reglue it via $\tau$. We call this surgered manifold \textit{a cork twist of $X$ along $(C,\tau)$} (even when this operation does not change the smooth structure).
\end{definition}
\begin{remark}
In this paper, we always assume that corks are contractible. (We did not assume this in the more general definition of~\cite{AY1}.) Note that Freedman's theorem tells us that every self-diffeomorphism of the boundary of $C$ extends to a self-homeomorphism of $C$ when $C$ is a compact contractible smooth $4$-manifold. 
\end{remark}
%\vspace{.1in}

\begin{definition}Let $W_n$ be the contractible smooth 4-manifold shown in Figure~$\ref{fig1}$. Let $f_n:\partial W_n\to \partial W_n$ be the obvious involution obtained by first surgering $S^1\times D^3$ to $D^2\times S^2$ in the interior of $W_n$, then surgering the other embedded $D^2\times S^2$ back to $S^1\times D^3$ (i.e. replacing the dot and ``0'' in Figure ~$\ref{fig1}$). Note that the diagram of $W_n$ comes from a symmetric link.
\begin{figure}[ht!]
\begin{center}
\includegraphics[width=1.1in]{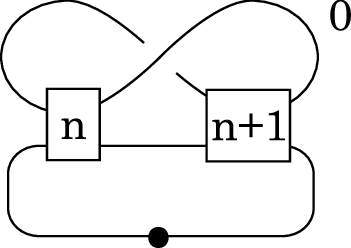}
\caption{$W_n$}
\label{fig1}
\end{center}
\end{figure}
\end{definition}
%\vspace{-1.15\baselineskip }
In \cite{AY1} a quick proof of the following theorem was given, by using the embedding theorem of Stein $4$-manifolds. 

\vspace{.05in} 

\begin{theorem}[{\cite[Theorem 2.5]{AY1}}]\label{th:cork}
For $n\geq 1$, the pair $(W_n, f_n)$ is a cork. 
\end{theorem}
\vspace{.05in} 

After the first draft of this paper, the following type corks are introduced in \cite{AKa}. 
\begin{definition}[\cite{AKa}] Let $C$ be a 4-dimensional oriented handlebody whose handlebody diagram consists of a dotted unknot $K_1$ and a $0$-framed unknot $K_2$. Let $L$ be the link in $S^3$ which consists of $K_1$ and $K_2$.  Suppose that $C$ satisfies the following conditions. 
\begin{itemize}
 \item [(1)] The link $L$ is symmetric. Namely, there exists a smooth isotopy of $S^3$ which exchanges the components $K_1$ and $K_2$ of $L$. \smallskip
 \item [(2)] The linking number of $K_1$ and $K_2$ is $\pm 1$. \smallskip
 \item [(3)] After converting the 1-handle notation of $C$ to the ball notation, $C$ becomes a Stein handlebody (i.e. the maximal Thurston-Bennequin number of $K_2$ with respect to the unique Stein fillable contact structure on $S^1\times S^2=\partial (S^1 \times D^3)$ is at least $+1$.)
\end{itemize}
Let $C'$ be the 4-manifold obtained by first surgering $S^1\times D^3$ to $D^2\times S^2$ inside $C$, then surgering the other embedded $D^2\times S^2$ back to $S^1\times D^3$ (i.e. exchanging the dot and $0$ in the handle picture of $C$). Since these surgeries were done in the interior of $C$, this operation gives a natural diffeomorphism $\varphi: \partial C \to \partial C'$. On the other hand, the condition $(1)$ gives a diffeomorphism $\psi: C\to C'$. Now let $\tau: \partial C\to \partial C$ be the involution defined by $\tau=(\psi|_{\partial C})^{-1}\circ \varphi$ ($\tau^2$ corresponds to the operation exchanging the dot and $0$ twice). The condition $(2)$ guarantees that $C$ is contractible. In this paper, we call such a pair $(C,\tau)$ \textit{a cork of Mazur type}. 
\end{definition}\vspace{.05in} 
Note that any cork of Mazur type indeed corks in the sense of Definition~\ref{sec:cork:def:cork}. This can be easily seen by applying the method of \cite{AM1} (cf.\ \cite{AY1}). The above $(W_n, f_n)$ $(n\geq 1)$ is clearly a cork of Mazur type. Note that $D^4$ cannot be a cork. The lemma below is sometimes useful. 
\vspace{.05in}
\begin{lemma}\label{cork:lem:S4} Let $(C, \tau)$ be a cork of Mazur type. Then its double $DC$ and the cork twist of $DC$ along $(C,\tau)$ are diffeomorphic to $S^4$.
\end{lemma}
\begin{proof}Consider natural handlebodies of these 4-manifolds. Since the unique dotted circle of the cork twist of $DC$ has a $0$-framed meridian, the latter claim is easily follows. Here note that the attaching circle $K_2$ of the unique 2-handle of $C$ is homotopic to the meridian of the dotted circle $K_1$ in the boundary $\partial(S^1\times D^3)$ of the sub 1-handlebody of $C$. This homotopy can be seen as changes of self-crossings of the attaching circle in the handlebody picture. Since $K_2$ has a $0$-framed meridian in the handlebody picture of $DC$, these crossing changes can be realized as handle slides. We can thus easily see that $DC$ is diffeomorphic to $S^4$. 
\end{proof}
%A way to prove this is as follows. Attach a emebed $W_n$ with a $2$-handle attached along the $-1$-framed meridian of the dotted circle. For details, see~\cite{AY1}. \cite{AY3} and also the following sections give different proofs of this theorem. 

%\vspace{.01in} 

\section{Infinitely many knotted corks}\label{sec:knotted cork}

Here we prove Theorem~\ref{th:knotting corks} by using the embedding theorem of Stein $4$-manifolds together with Fintushel-Stern knot surgery. For simplicity, we give a proof for the $(W_1,f_1)$ cork. The same argument holds for any cork of Mazur type. \medskip 

\begin{definition}\label{def:S in knotted corks}
Let $S$ be the compact $4$-manifold with boundary in Figure~\ref{fig2}. Note that $S$ disjointly contains a cusp neighborhood and $W_1$. 
\begin{figure}[ht!]
\begin{center}
\includegraphics[width=2.6in]{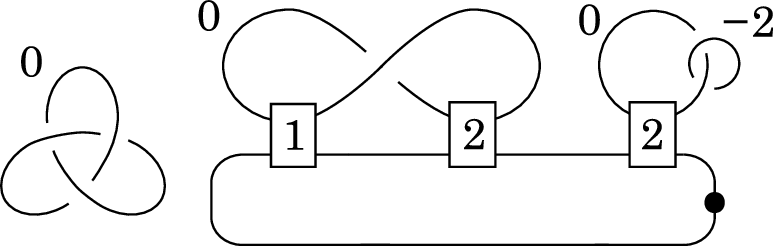}
\caption{$S$}
\label{fig2}
\end{center}
\end{figure}
\end{definition}

%\begin{lemma}The $4$-manifold $S$ is a simply connected compact Stein $4$-manifold. 
%\end{lemma}
%\begin{proof}This follows from Figure~?. 
%\end{proof}
The embedding theorem of Stein $4$-manifold gives the minimal symplectic closed 4-manifold below:  

\begin{proposition}\label{prop: widetilde{S} of knotting corks}
There exists a simply connected, minimal, closed, symplectic $4$-manifold $\widetilde{S}$ with  the following properties: 

\begin{itemize}

\item[(1)] $b_2^+(\widetilde{S})>1$; \smallskip

\item[(2)] $\widetilde{S}$ contains the $4$-manifold $S$;\smallskip

\item [(3)] A naturally embedded torus in the cusp neighborhood of $S$ represents a non-zero second homology class of $\widetilde{S}$.

\end{itemize}

\end{proposition}

\begin{proof} Change the diagram of $S$ into the Legendrian diagram in Figure~\ref{fig3} (in particular change the notation of the $1$-handle from the dotted unknot to pair of balls). Here the coefficients of the $2$-handles denote the contact framings. Then attach a $2$-handle to $S$ along the dotted meridian in Figure~\ref{fig3} with contact $-1$ framing. Since each framing is contact $-1$ framing, this new handlebody is also a compact Stein $4$-manifold. Hence Theorem~\ref{th:closing of Stein} gives us a simply connected, minimal, closed, symplectic $4$-manifold $\widetilde{S}$ with $b_2^+(\widetilde{S})>1$ which contains this handlebody. The $4$-manifold $\widetilde{S}$ has the property $(3)$, because the torus in the cusp neighborhood algebraically intersects a sphere with the self-intersection number $-2$. 
\begin{figure}[ht!]
\begin{center}
\includegraphics[width=3.0in]{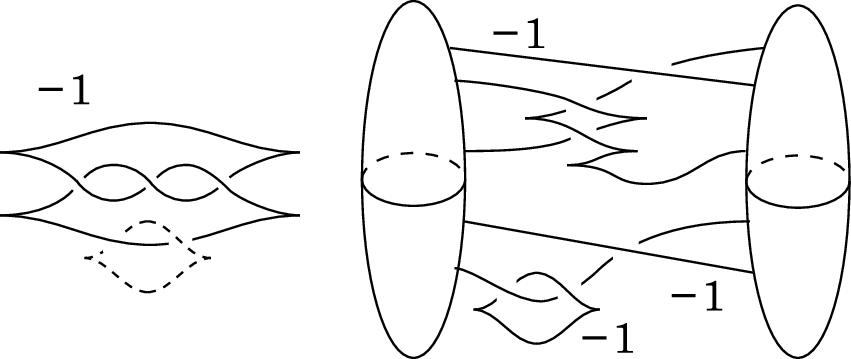}
\caption{Legendrian diagram of $S$ with contact framings}
\label{fig3}
\end{center}
\end{figure}
%\begin{figure}[ht!]
%\begin{center}
%\includegraphics[width=2.7in]{knotted_stein_eps2eps.eps}
%\caption{Ledgendrian diagram of $S$}
%\label{knotted_stein}
%\end{center}
%\end{figure}
\end{proof}

Though the rest of argument in this section is almost the same as \cite{AY3}, we proceed the proof for the completeness.

\vspace{.05in}

\begin{definition}\label{def:X of knotting corks}  \hspace{-.1in}
%$(1)$ Let $\widetilde{S}$ be a simply connected, minimal, closed, symplectic $4$-manifold with $b_2^+>1$ which contains the $4$-manifold $S$. \smallskip \\

\begin{itemize}

\item[(1)] Let $X$ be the cork twist of $\widetilde{S}$ along $(W_1, f_1)$, where this copy of $W_1$ is the one contained in $S$. Note that $X$ contains a cusp neighborhood because the copy of $W_1$ in $S$ is disjoint from the cusp neighborhood of $S$.\smallskip 

\item [(2)] Let $K$ be a knot in $S^3$,  and $X_K$ denote the manifold obtained by the  (Fintushel-Stern's) knot surgery operation with $K$ in the cusp neighborhood of $X$  (\cite{FS2}) . 
\smallskip 

\item[(3)] Let $\widetilde{S}_K$ be the knot surgered $\widetilde{S}$ with $K$ in the cusp neighborhood of $\widetilde{S}$. 

\end{itemize}

\end{definition}\smallskip

The corollary below clearly follows from Proposition~\ref{prop: widetilde{S} of knotting corks}, Definition~\ref{def:X of knotting corks} and the diagram of $S$ in Figure~\ref{fig2}. See also Figure~\ref{fig4}. 

%\vspace{.05in}

\vspace{.1in}

\begin{corollary}\label{cor:1_X_K}   \hspace{-.1in}
\begin{itemize}
\item[(1)] The copy of $W_1$ in $X$ $($given in Definition~$\ref{def:X of knotting corks}.$$)$ is disjoint from the cusp neighborhood of $X$. \smallskip 

\item[(2)] $X$ splits off $S^2\times S^2$ as a connected summand. Consequently,  the Seiberg-Witten invariant of $X$ vanishes. Furthermore, the cusp  neighborhood of $X$ is disjoint from $S^2\times S^2$, in this connected sum decomposition of $X$. \smallskip 

\item[(3)]  $\widetilde{S}_K$ is obtained from $X_K$ by a cork twist along $(W_1,f_1)$.  In particular, $X_K$  is homeomorphic to $\widetilde{S}_K$. \smallskip 

\end{itemize}

\end{corollary}

%\vspace{.1in}
\newpage

\begin{corollary}\label{cor:2_X_K}    \hspace{-.1in}
\begin{itemize}
\item[(1)] $X_K$ is diffeomorphic to $X$. In particular, the Seiberg-Witten invariant of $X_K$ vanishes. \smallskip 

\item[(2)] For each knot $K$ in $S^3$, $\widetilde{S}_K$ is obtained from $X$ by a cork twist along $(W_1,f_1)$.  \smallskip 

\item[(3)] If $K$ in $S^3$ has non-trivial Alexander polynomial, then $\widetilde{S}_K$ is homeomorphic but not diffeomorphic to $X$, in particular $(W_1,f_1)$ is a cork of $X$.  \smallskip 

\end{itemize}

\end{corollary}%\smallskip 

\begin{proof}The claim (1) follows from 
Corollary~\ref{cor:1_X_K}.(2), the definition of $X_K$ and the stabilization theorem of knot surgery by the first author~\cite{A4} and Auckly~\cite{Au}.  Corollary~\ref{cor:1_X_K}.(3) thus shows the claim (2). Since the Seiberg-Witten invariants of $\widetilde{S}_K$ does not vanish (Fintushel-Stern~\cite{FS2}), the claim (3) follows from the claim (1). 
%Note that a naturally embedded torus into the cusp neighborhood of $\widetilde{S}$ represent a non-zero homology class because it algebraically intersects a sphere with self-intersection number $-2$. 
\end{proof}
\begin{figure}[ht!]
\begin{center}
\includegraphics[width=3.4in]{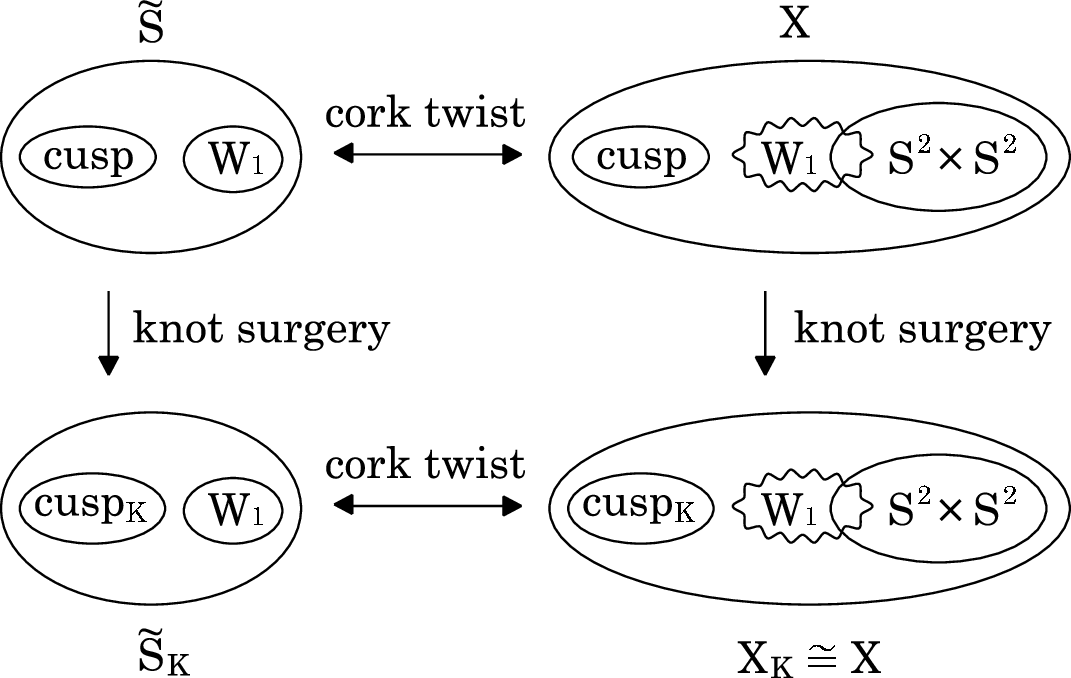}
\caption{relation between $\widetilde{S}, X, \widetilde{S}_K$ and $X_K$}
\label{fig4}
\end{center}
\end{figure}
Now we can easily prove Theorem~\ref{th:knotting corks}. 

\begin{proof}[Proof of Theorem~$\ref{th:knotting corks}$]
%\begin{proof}[of Theorem~$\ref{th:knotting corks}$]
Let $X_0:=X$, and $K_n$ $(n\geq 1)$ be knots in $S^3$ with mutually different non-trivial Alexander polynomials. Define $X_n = \widetilde{S}_{K_n}$. Then the claim easily follows from Corollary~\ref{cor:2_X_K} and the Fintushel-Stern's formula (\cite{FS2}) of the Seiberg-Witten invariant of knot surgered manifolds. This gives a proof for the $(W_1,f_1)$ cork. Clearly the same argument holds for any cork of Mazur type. 
\end{proof}

%\begin{remark}\label{rem:knotted corks}
%We proved Theorem~\ref{th:knotting corks} for the cork $(W_1,f_1)$. One can similarly prove Theorem~\ref{th:knotting corks} for many other corks, including $(W_n,f_n)$. 
%\end{remark}

%%%%%%%%%%%%%%%%%%%%%%%%%%%%%%%%%%%%%%%%%%%%%%%%%%%%
\section{Rational blowdown}
In this section we review the rational blowdown introduced 
by Fintushel-Stern \cite{FS1}, and recall relations between rational blowdowns and corks.\medskip 

Let $C_p$ and $B_p$ be the smooth $4$-manifolds defined by handlebody diagrams in Figure~\ref{fig5}, and $u_1,\dots,u_{p-1}$ elements of $H_2(C_p;\mathbf{Z})$ given by corresponding $2$-handles in the figure such that $u_i\cdot u_{i+1}=+1$ $(1\leq i\leq p-2)$.
The boundary $\partial C_p$ of $C_p$ is diffeomorphic to the lens space $L(p^2,p-1)$, and also diffeomorphic to the boundary $\partial B_p$ of $B_p$. 
\begin{figure}[ht!]
\begin{center}
\includegraphics[width=3.8in]{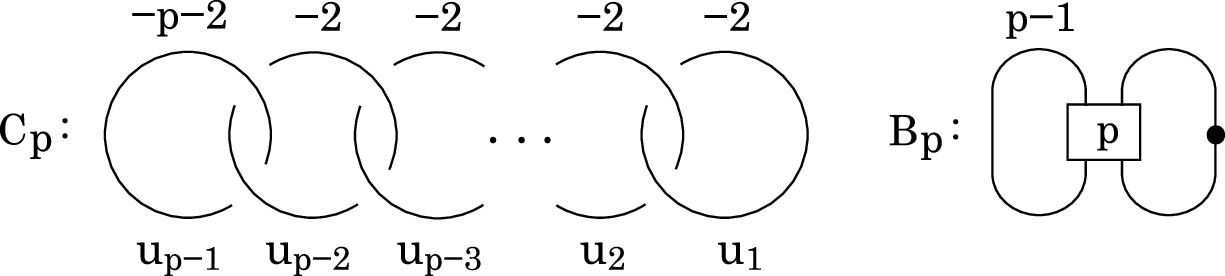}
\caption{}
\label{fig5}
\end{center}
\end{figure}

Suppose that $C_p$ embeds in a smooth $4$-manifold $Z$. 
Let $Z_{(p)}$ be the smooth $4$-manifold obtained from $Z$ by removing $C_p$ and gluing $B_p$ along the boundary. The smooth $4$-manifold $Z_{(p)}$ is called the rational blowdown of $Z$ along $C_p$. Note that $Z_{(p)}$ is uniquely determined up to diffeomorphism by a fixed pair $(Z,C_p)$ (see Fintushel-Stern~\cite{FS1}). 
This operation preserves $b_2^+$, decreases $b_2^-$, may create torsion in the first homology group. \medskip 

Rational blowdowns have the following relations with corks. 

\begin{theorem}[\cite{AY3}, see also \cite{AY1}]\label{th:cork and rbd}Let $D_p$ be the smooth $4$-manifold in Figure~\ref{fig6} (notice that $D_p$ is $C_p$ with two $2$-handles attached). 
Suppose that a smooth $4$-manifold $Z$ contains $D_{p}$. Let $Z_{(p)}$ be the rational blowdown of $Z$ along the copy of $C_p$ contained in $D_p$. 
Then the submanifold $D_p$ of $Z$ contains $W_{p-1}$ such that $Z_{(p)}\# (p-1)\overline{\mathbf{C}\mathbf{P}^2}$ is obtained from $Z$ by the cork twist along $(W_{p-1},f_{p-1})$.
\begin{figure}[ht!]
\begin{center}
\includegraphics[width=2.15in]{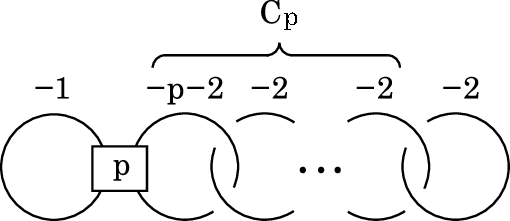}
\caption{$D_p$}
\label{fig6}
\end{center}
\end{figure}
\end{theorem}
%%%%%%%%%%%%%%%%%%%%%%%%%%%%%%%%%%%%%%%%%%%%%%%%%%%%%%%%%

\section{Seiberg-Witten invariants}

\vspace{.05in} 

In this section, we briefly review basic facts about the Seiberg-Witten invariants. 
For more details, see, for example, Fintushel-Stern~\cite{FS3}, Gompf-Stipsicz~\cite{GS}, Ozbagci-Stipsicz~\cite{OS1}.\medskip 

\vspace{.05in} 

Suppose that $Z$ is a simply connected closed smooth $4$-manifold with $b_2^+(Z)>1$ and odd. Let $\mathcal{C}(Z)$ be the set of 
characteristic elements of $H^2(Z;\mathbf{Z})$. %Fix a homology orientation on $X$, that is, orient $H^2_+(X;\mathbf{R})$. 
Then the Seiberg-Witten invariant $SW_{Z}: \mathcal{C}(Z)\to \mathbf{Z}$ is defined. 
Let $e(Z)$ and $\sigma(Z)$ be the Euler characteristic and the signature of $Z$, respectively, and $d_Z(K)$ the even integer defined by $d_Z(K)=\frac{1}{4}(K^2-2e(Z)-3\sigma(Z))$ 
for $K\in \mathcal{C}(Z)$. 
%It is known that if $SW_{X}(K)\neq 0$, then $d_X(K)\ge 0$. 
If $SW_Z(K)\neq 0$, then $K$ is called a Seiberg-Witten basic class of $Z$. It is known that if $K$ is a Seiberg-Witten basic class of $Z$, then $-K$ is also a Seiberg-Witten basic class of $Z$. We denote $\beta(Z)$ as the set of the Seiberg-Witten basic classes of $Z$.  The blow up formula is as follows:

\vspace{.05in} 

\begin{theorem}[Witten~\cite{W}, cf.\ Gompf-Stipsicz~\cite{GS}]\label{th:SW of E(n)} 
Suppose that $Z$ is a simply connected closed smooth $4$-manifold with $b_2^+(Z)>1$. If $\beta(Z)$ is not empty, then
\begin{equation*}
\beta(Z\#n\overline{\mathbf{C}\mathbf{P}^2})=\{K\pm E_1\pm E_2\pm \dots \pm E_n\mid K\in \beta(Z)\}. 
\end{equation*}
Here $E_1,E_2,\dots,E_n$ denotes the standard orthogonal basis of $H^2(n\overline{\mathbf{C}\mathbf{P}^2};\mathbf{Z})$ such that $E^2_i=-1$ $(1\leq i\leq n)$. 
%\begin{theorem}\label{th:SW of E(n)}For $n\geq 2$, \smallskip \\
%$(1)$ $\beta(E(n))=\{k\cdot PD(T)\mid k\equiv 0\; (\textnormal{mod}\; 2), \lvert k\rvert \leq n-2 \}$$;$\smallskip 

%$(2)$ $\beta(E(n)\# m\overline{\mathbf{C}\mathbf{P}^2})=\{k\cdot PD(T)\pm E_1\pm E_2\pm \dots \pm E_m\mid k\equiv 0\; (\textnormal{mod}\; 2), \lvert k\rvert \leq n-2 \}$. Here $T$ denotes the class of a regular fiber of $E(n)$ in $H_2(E(n);\mathbf{Z})$, and $E_1,E_2,\dots,E_m$ denotes the standard basis of $H^2(m\overline{\mathbf{C}\mathbf{P}^2};\mathbf{Z})$. 
\end{theorem}

If every Seiberg-Witten basic class $K$ of $Z$ satisfies $d_Z(K)=0$, then the $4$-manifold $Z$ is called of simple type. For example, it is known that every closed symplectic $4$-manifold with $b_2^+>1$ has non-vanishing Seiberg-Witten invariant and is of simple type. For such a 4-manifold, the following adjunction inequality holds:
\begin{theorem}[Ozsv\'{a}th-Szab\'{o}~\cite{OzSz}, cf.\ Ozbagci-Stipsicz~\cite{OS1}]
Suppose that $Z$ is a simply connected closed smooth $4$-manifold of simple type with $b_2^+(Z)>1$, and that $\Sigma \subset Z$ is a smoothly embedded, oriented, connected closed surface of genus $g>0$. Let $[\Sigma]$ be the second homology class of $Z$ represented by the embedded surface $\Sigma$. Then, for every Seiberg-Witten basic class $K$ of $X$, the following adjunction inequality  holds: 
\begin{equation*}
[\Sigma]^2+\lvert \langle K, [\Sigma]\rangle \rvert \leq 2g-2. 
\end{equation*}

\end{theorem}
%\smallskip  
%
%\subsection{Seiberg-Witten invariants of symplectic 4-manifolds}
%\begin{theorem}[cf. Gompf-Stipsicz~\cite{GS}]Let $(Z,\omega)$ be a simply connected, minimal, closed, symplectic $4$-manifold with $b_2^+(Z)>1$. Then $(Z,\omega)$ has the following properties:\smallskip \\
%$(1)$ $c_1(X,\omega)^2\geq 0$;\smallskip \\
%$(2)$ If $K$ is a Seiberg-Witten basic class of $Z$,  then $(c_1(Z,\omega)-K)^2\neq -4$.
%\end{theorem}
%\begin{proposition}[cf. Lisca-Mati\'{c}~\cite{LM2}]
%Let $(Z,\omega)$ be a simply connected, minimal, closed, symplectic $4$-manifold with $b_2^+(Z)>1$. Suppose that an element $\alpha$ of $H_2(Z;\mathbf{Z})$ is represented by a smoothly embedded sphere and that $\alpha^2=-2$. Then, every Seiberg-Witten basic class $K$ of $Z$ satisfies $\langle K, \alpha \rangle=0$.
%\end{proposition}
%\begin{proof}We use the theorem below. 
%\begin{theorem}[Friedman-Morgan]Let $Z$ be an oriented smooth $4$-manifold. Suppose that an element $\alpha$ of $H_2(Z;\mathbf{Z})$ is represented by a smoothly embedded sphere and that $\alpha^2=-2$. Then there exists an orientation preserving self-diffeomorphism $F_{\alpha}$ of $Z$ such that $F^*_{\alpha}(L)=L+\langle L, \alpha \rangle\cdot PD(\alpha)$ for every element $L$ of $H^2(Z;\mathbf{Z})$.
%\end{theorem}
%Suppose that a Seiberg-Witten basic class $K$ of $Z$ satisfies $\langle K, \alpha \rangle\neq 0$. 
%Since $F_{\alpha}$ is an orientation preserving self-diffeomorphism of $Z$, the element $F^*_{\alpha}(K)=K+\langle K, \alpha \rangle\cdot PD(\alpha)$ is a Seiberg-Witten basic class of $Z$. Thus the 
%\end{proof}

\subsection{Seiberg-Witten invariants of rational blowdowns}
We here recall the change of the Seiberg-Witten invariants by rationally blowing down. Let $Z$ be a simply connected closed smooth $4$-manifold with $b_2^+(Z)>1$ and odd. Suppose that $Z$ contains a copy of $C_p$. Let $Z_{(p)}$ be the rational blowdown of $Z$ along the copy of $C_p$. Assume that $Z_{(p)}$ is simply connected. The following theorems are obtained by Fintushel-Stern~\cite{FS1}.

\begin{theorem}[Fintushel-Stern \cite{FS1}]\label{th:SW1}
For every element $K$ of $\mathcal{C}(Z_{(p)})$, there exists an element $\tilde{K}$ of $\mathcal{C}(Z)$ such that 
$K\rvert _{Z_{(p)}-B_{p}}=\tilde{K}\rvert _{Z-C_{p}}$ and 
$d_{Z_{(p)}}(K)=d_Z(\tilde{K})$. Such an element $\tilde{K}$ of $\mathcal{C}(Z)$ is called a lift of $K$.
\end{theorem}

\begin{theorem}[Fintushel-Stern \cite{FS1}]\label{th:SW2}
If an element $\tilde{K}$ of $\mathcal{C}(Z)$ is a lift of some element $K$ of $\mathcal{C}(Z_{(p)})$, then $SW_{Z_{(p)}}(K)=SW_{Z}(\tilde{K})$. 
\end{theorem}

\begin{theorem}[Fintushel-Stern \cite{FS1}, cf.~Park \cite{P1}]\label{th:SW3}
If an element $\tilde{K}$ of $\mathcal{C}(Z)$ satisfies that $(\tilde{K}\rvert _{C_p})^2=1-p$ and 
$\tilde{K}\rvert _{\partial C_p}=mp\in \mathbf{Z}_{p^2}\cong H^2(\partial C_p;\mathbf{Z})$ 
with $m\equiv p-1\pmod 2$, then there exists an element $K$ of $\mathcal{C}(Z_{(p)})$ such that $\tilde{K}$ is a lift of $K$. 
\end{theorem}

\begin{corollary}\label{cor:SW4}
If an element $\tilde{K}$ of $\mathcal{C}(Z)$ satisfies $\tilde{K}(u_1)=\tilde{K}(u_2)=\dots=\tilde{K}(u_{p-2})=0$ and $\tilde{K}(u_{p-1})=\pm p$, then $\tilde{K}$ is a lift of some element $K$ of $\mathcal{C}(Z_{(p)})$. 
\end{corollary}

%%%%%%%%%%%%%%%%%%%%%%%%%%%%%%%%%%%%%%%%%%%%%%%%%%%%%%%%

\section{Disjointly embedded corks}\label{sec:construction}
In this section we prove Theorem~\ref{th:disjoint corks} and Corollary~\ref{cor:involutions of corks} by using the embedding theorem of Stein 4-manifolds together with Finushel-Stern's rational blowdown.\medskip

\begin{definition} Let $\widetilde{D}_p$ $(p\geq 2)$ be the compact $4$-manifold with boundary in Figure~\ref{fig7} (i.e. the blow down of $D_p$). Note that the left most knot in the figure is a $(p+1,p)$ torus knot.  Define $\widetilde{D}(p_1,p_2,\dots,p_n)$ as the boundary sum of $\widetilde{D}_{p_1},\widetilde{D}_{p_2},\dots,\widetilde{D}_{p_n}$. We denote by $D(p_1,p_2,\dots,p_n)$ the boundary sum of $D_{p_1},D_{p_2},\dots,D_{p_n}$. %Note that $D(p_1,p_2,\dots,p_n)$ is a blow up of $\widetilde{D}(p_1,p_2,\dots,p_n)$ $n$ times .
\begin{figure}[ht!]
\begin{center}
\includegraphics[width=3.5in]{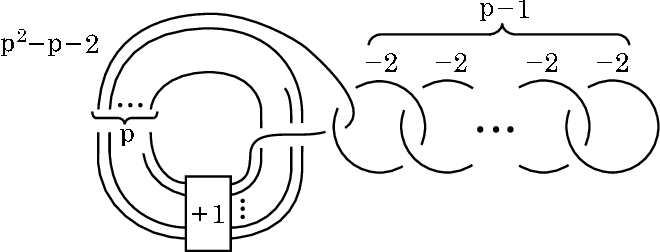}
\caption{$\widetilde{D}_p$}
\label{fig7}
\end{center}
\end{figure}
\end{definition}
%\begin{lemma}For $p_1,p_2,\dots p_n\geq 2$, the $4$-manifold $S(p_1,p_2,\dots,p_n)$ is a simply connected compact Stein $4$-manifold. 
%\end{lemma}
%\begin{proof}
%\end{proof}
\begin{figure}[ht!]
\begin{center}
\includegraphics[width=4.5in]{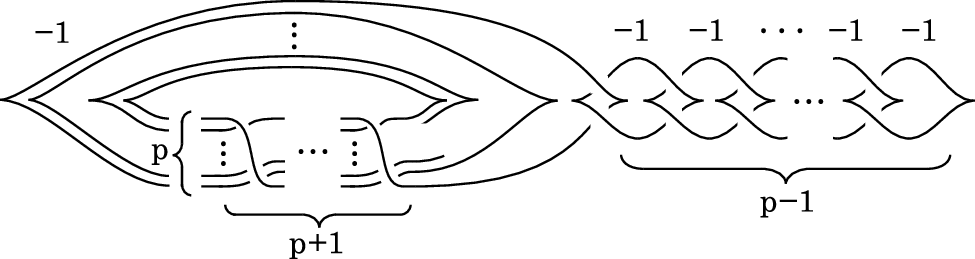}
\caption{Legendrian diagram of $\widetilde{D}_p$ with contact framings}
\label{fig8}
\end{center}
\end{figure}

\begin{proposition}\label{prop: S}
For each $n\geq 1$ and each $p_1,p_2,\dots,p_n\geq 2$, there exists a simply connected, minimal, closed, symplectic $4$-manifold $S(p_1,p_2,\dots,p_n)$ with the following properties: \smallskip 

$(1)$ $b_2^+(S(p_1,p_2,\dots,p_n))>1$; \smallskip 

$(2)$ The $4$-manifold $S(p_1,p_2,\dots,p_n)$ contains $\widetilde{D}(p_1,p_2,\dots,p_n)$;\smallskip 

$(3)$ Every Seiberg-Witten basic class $K$ of $S(p_1,p_2,\dots,p_n)$ satisfies $\langle K, \alpha \rangle=0$ for all $\alpha\in \iota_*H_2(\widetilde{D}(p_1,p_2,\dots,p_n);\mathbf{Z})$. Here $\iota_*$ denotes the homomophism induced by the inclusion $\iota: \widetilde{D}(p_1,p_2,\dots,p_n)\hookrightarrow  S(p_1,p_2,\dots,p_n)$.
\end{proposition}

\begin{proof}$(1)$ and $(2)$. 
Change the diagram of $\widetilde{D}(p_1,p_2,\dots,p_n)$ into the Legendrian diagram as in Figure~\ref{fig8}. Then, for each unknot, attach a $2$-handle along a contact $-1$-framed trefoil knot, as in Figure~\ref{fig9}. Since every 2-handle has contact $-1$-framings, this new handlebody is a compact Stein $4$-manifold. The embedding theorem of Stein $4$-manifolds thus gives a simply connected, minimal, closed, symplectic $4$-manifold $S(p_1,p_2,\dots,p_n)$ with $b_2^+(S(p_1,p_2,\dots,p_n))>1$ which contains this handlebody. The claims $(1)$ and $(2)$ hence follows.

$(3)$. Let $u$ (resp.\ $v$) be the element of $H_2(S(p_1,p_2,\dots,p_n);\mathbf{Z})$ 
%represented by the sphere (resp.\ torus)
 given by a contact $-1$-framed unknot (resp.\ a contact $-1$-framed trefoil knot), as in Figure~\ref{fig9}. We can easily check, by a handle slide, that $u+v$ is represented by a torus with the self-intersection number $0$. Let $K$ be a Seiberg-Witten basic class of $S(p_1,p_2,\dots,p_n)$. Since $v$ and $u+v$ are represented by tori with self-intersection numbers $0$, the adjunction inequality gives $\langle K, v \rangle=0$ and $\langle K, u+v \rangle=0$. We thus have $\langle K, u \rangle=0$. Let $w_i$ be the element of $H_2(S(p_1,p_2,\dots,p_n);\mathbf{Z})$ given by the contact $-1$-framed $(p_i+1,p_i)$ torus knot of $\widetilde{D}_{p_i}$. Note that $w_i$ is represented by a genus $\frac{p_i(p_i-1)}{2}$ surface with the self-intersection number $p^2_i-p_i-2$. We can also easily check, by the adjunction inequality, that $\langle K, w_i \rangle=0$ $(1\leq i\leq n)$.  Now $(3)$  follows from the fact that $\iota_*H_2(\widetilde{D}(p_1,p_2,\dots,p_n);\mathbf{Z})$ is generated by these classes $u$ and $w_i$. 
\end{proof}
\begin{figure}[ht!]
\begin{center}
\includegraphics[width=1.4in]{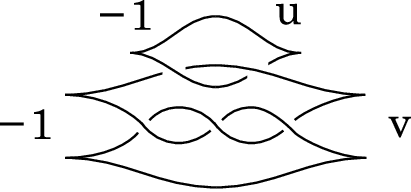}
\caption{}
\label{fig9}
\end{center}
\end{figure}

%\begin{definition}
%For $p_1,p_2,\dots p_n\geq 2$, let $\widetilde{S}(p_1,p_2,\dots,p_n)$ be a simply connected, minimal, closed, symplectic $4$-manifold which contains $S(p_1,p_2,\dots,p_n)$. 
%\end{definition}
Let $e_1,e_2,\dots,e_n$ be the standard basis of $H_2(n\overline{\mathbf{C}\mathbf{P}^2};\mathbf{Z})$ such that $e_i^2=-1$ $(1\leq i\leq n)$ and $e_i\cdot e_j=0$ $(i\neq j)$. Then the proposition above together with the blow up formula immediately gives the following corollary.

\begin{corollary}\label{cor:disjoint C_p}
$(1)$ For each $p_1,p_2,\dots,p_n\geq 2$, the $4$-manifold 
$S(p_1,p_2,\dots,p_n)\# n\overline{\mathbf{C}\mathbf{P}^2}$ contains $D(p_1,p_2,\dots,p_n)$ as in Figure~\ref{fig10}. Here, $u^{(i)}_j$ $(1\leq i\leq n,\,\, 0\leq j\leq p_i-1)$ in the figure denotes the second homology class given by corresponding $2$-handle, and $e_i$ $(1\leq i\leq n)$ represents the homology class given by the corresponding $2$-handle. We orient $e_i$ and $u^{(i)}_j$ so that $e_i\cdot u^{(i)}_{p_i-1}=p_i$ and $u^{(i)}_j\cdot u^{(i)}_{j+1}=1$ $(1\leq i\leq n,\,\, 0\leq j\leq p_i-2)$. 
 
\noindent $(2)$ Every Seiberg-Witten basic class $K$ of $S(p_1,p_2,\dots,p_n)\# n\overline{\mathbf{C}\mathbf{P}^2}$ satisfies $\langle K, u^{(i)}_{p_i-1} \rangle =\langle K, -p_ie_i \rangle =\pm p_i$ and $\langle K, u^{(i)}_{j} \rangle=0$ $(1\leq i\leq n,\,\, 0\leq j\leq p_i-2)$. 
%Here, $u^{(i)}_{j}$ denotes the element of $H_2(C_{p_i};\mathbf{Z})$ corresponding to the element $u_j$ of $H_2(C_{p};\mathbf{Z})$. 
\begin{figure}[ht!]
\begin{center}
\includegraphics[width=2.4in]{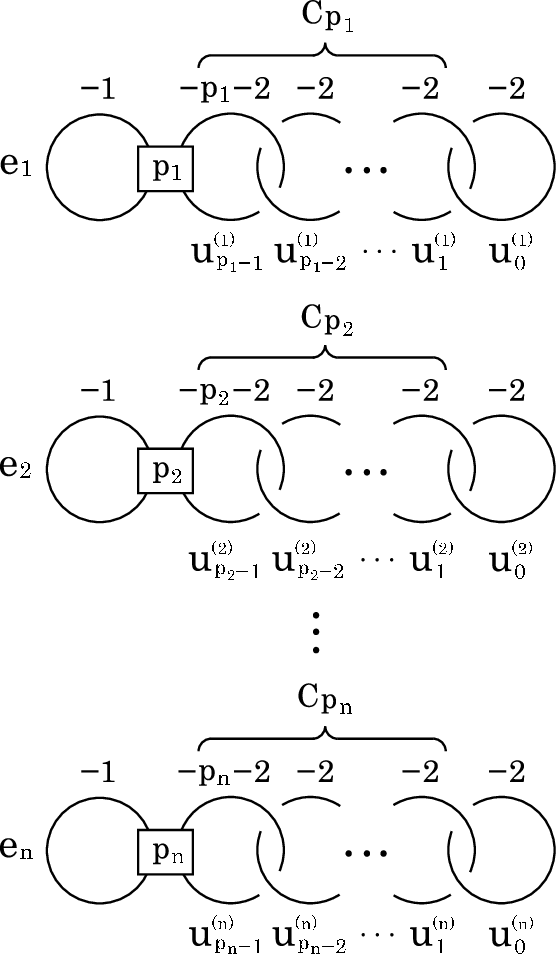}
\caption{the submanifold $D(p_1,p_2,\dots,p_n)$ of $S(p_1,p_2,\dots,p_n)\# n\overline{\mathbf{C}\mathbf{P}^2}$% $(p_1,p_2,\dots,p_n\geq 2)$
}
\label{fig10}
\end{center}
\end{figure}
\end{corollary}
%\vspace{-1.5\baselineskip }
%%%%%%%%%%%%%%%%%%%%%%%%%%%%%%%%%%%%%%%%%%%%%%inyou

\begin{definition}\label{def:disjoint cork}
$(1)$ Define $X_0:=S(p_1,p_2,\dots,p_n)\# n\overline{\mathbf{C}\mathbf{P}^2}$. Let $X_i'$ $(1\leq i\leq n)$ be the rational blowdown of $X_0$ along the copy of $C_{p_i}$ in Figure~\ref{fig10}. Put $X_i:=X_i'\# (p_i-1)\overline{\mathbf{C}\mathbf{P}^2}$.\medskip 

\noindent $(2)$ For $k_1,k_2,\dots,k_n\geq 1$, let $W(k_1,k_2,\dots,k_n)$ be the boundary sum $W_{k_1}\natural W_{k_2}\natural \cdots \natural W_{k_n}$. Figure~\ref{fig11} is a diagram of $W(k_1,k_2,\dots,k_n)$. Let $f^i(k_1,k_2,\dots,k_n)$ be the involution on the boundary $\partial W(k_1,k_2,\dots,k_n)$ obtained by replacing the dot and zero of the component of $W_{k_i}$.%
\begin{figure}[ht!]
\begin{center}
\includegraphics[width=4.5in]{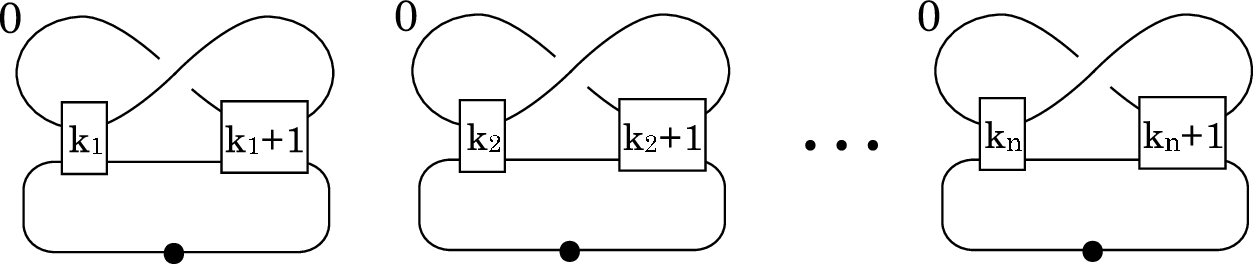}
\caption{$W(k_1,k_2,\dots,k_n)$}
\label{fig11}
\end{center}
\end{figure}
\end{definition}
One can easily prove the lemma below by checking Thurston-Bennequin numbers of 2-handles. For a proof, see \cite{AY3}.
\begin{lemma}[\cite{AY3}]
For each $k_1,k_2,\dots,k_n\geq 1$, the manifold $W(k_1,k_2,\dots,k_n)$ is a compact contractible Stein $4$-manifold.
\end{lemma}

\begin{proposition}\label{prop:X_i}
$(1)$ The $4$-manifold $X_0$ contains mutually disjoint copies of $W_{p_1-1}, W_{p_2-1}, \dots, W_{p_n-1}$ such that, for each $i$, the $4$-manifold $X_i$ is obtained from $X_0$ by the cork twist along $(W_{p_i-1}, f_{p_i-1})$.\medskip

$(2)$ The $4$-manifold $X_0$ contains a fixed copy of $W(p_1-1,p_2-1,\dots,p_n-1)$ such that, for each $i$, the $4$-manifold $X_i$ is obtained from $X_0$ by the cork twist along $(W(p_1-1,p_2-1,\dots,p_n-1),\, f^i(p_1-1,p_2-1,\dots,p_n-1))$.
\end{proposition}
\begin{proof}
Corollary~\ref{cor:disjoint C_p} and Theorem~\ref{th:cork and rbd} clearly show the claims $(1)$ and (2). 
\end{proof}

\begin{remark}We here correct a misprint in \cite{AY3}. 
Proposition~5.4.(1) of \cite{AY3} should be changed as in Proposition~\ref{prop:X_i}.(1) of this paper. However, the claim itself of Proposition~5.4.(1) of \cite{AY3} is correct, because we can easily replace $(W_{p-1},f_{p-1})$ with $(W_{p},f_p)$ in Theorem~\ref{th:cork and rbd} of this paper. See the proof of the theorem given in~\cite{AY3}. 
\end{remark}
%%%%%%%%%%%%%%%%%%%%%%%%%%%%%%%%%%%%%%%%%%%%%%%%%%%%%%%%%%%%%%%%%%%%

\subsection{Computation of SW invariants}In this subsection, we complete the proofs of Theorem~\ref{th:disjoint corks} and Corollary~\ref{cor:involutions of corks} by computing the Seiberg-Witten invariants of the $4$-manifolds $X_i$ $(0\leq i\leq n)$ in Definition~\ref{def:disjoint cork}.

\begin{lemma}\label{lem:restriction}Fix an integer $i\in\{1,2,\dots,n\}$. If Seiberg-Witten basic classes $K$ and $K'$ of $X_0$ satisfy $K\neq K'$, then restrictions $K|_{X_0-C_{p_i}}$ and ${K'}|_{X_0-C_{p_i}}$ are not equal to each other. 
\end{lemma}

\begin{proof}Define an element $\alpha$ of $H_2(X_0;\mathbf{Z})$ by 
\begin{equation*}
\alpha=e_i+u^{(i)}_{p_i-1}+2u^{(i)}_{p_i-2}+3u^{(i)}_{p_i-3}+\dots+(p_i-1)u^{(i)}_{1}+p_iu^{(i)}_{0}.
\end{equation*}
We get $\langle K,\,\alpha \rangle =\langle K,\,(1-p_i)e_i\rangle$ and $\langle K',\,\alpha \rangle=\langle K',\,(1-p_i)e_i \rangle$, because Corollary~\ref{cor:disjoint C_p} gives $\langle K, u^{(i)}_{p_i-1} \rangle =\langle K, -p_ie_i \rangle$, $\langle K', u^{(i)}_{p_i-1} \rangle =\langle K', -p_ie_i \rangle$ and $\langle K,\,u^{(i)}_{j} \rangle = \langle K',\,u^{(i)}_{j} \rangle=0$ $(0\leq j\leq p_i-2)$ . 
Corollary~\ref{cor:disjoint C_p} implies 

\begin{equation*}
\alpha \cdot u^{(i)}_{1}=\alpha \cdot u^{(i)}_{2}=\dots=\alpha \cdot u^{(i)}_{p_i-1}=0. 
\end{equation*}
Since $u^{(i)}_0$ satisfies 
\begin{equation*}
u^{(i)}_0\cdot u^{(i)}_1=1\;\, \textnormal{and}\;\, u^{(i)}_0\cdot u^{(i)}_2=u^{(i)}_0\cdot u^{(i)}_3=\dots=u^{(i)}_0\cdot u^{(i)}_{p_i-1}=0, 
\end{equation*}
Lemma~5.1 of \cite{Y1} shows that $\alpha$ is an element of $\iota_*H_2(X_0-C_{p_i};\mathbf{Z})$, where $\iota_*$ is the homomorphism induced by the inclusion $\iota:X_0-C_{p_i}\hookrightarrow X_0$. Lemma~5.1 of \cite{Y1} also gives $H_1(X_0-C_{p_i};\mathbf{Z})=0$. Thus Mayer-Vietoris exact sequence for $C_{p_i}\cup (X_0-C_{p_i})=X_0$ is as follows: 
\begin{equation*}
0\to H_2(C_{p_i};\mathbf{Z})\oplus H_2(X_0-C_{p_i};\mathbf{Z})\rightarrow H_2(X_0;\mathbf{Z})\rightarrow \mathbf{Z}_{{p^2_i}}\to 0.
\end{equation*}

The case where $K|_{C_{p_i}}=K'|_{C_{p_i}}$: In this case, the exact sequence above together with the universal coefficient theorem for $X_0$ implies $K|_{X_0-C_{p_i}}\neq {K'}|_{X_0-C_{p_i}}$. 

The case where $K|_{C_{p_i}}\neq K'|_{C_{p_i}}$: In this case, we have $\langle K,\,u^{(i)}_{p_i-1} \rangle \neq \langle K',\,u^{(i)}_{p_i-1} \rangle$, because $\langle K,\,u^{(i)}_{j} \rangle = \langle K',\,u^{(i)}_{j} \rangle=0$ for $0\leq j\leq p_i-2$.  
%Proposition~\ref{prop: S} and Corollary~\ref{cor:disjoint C_p} thus implies 
We thus get $\langle K,\,-p_ie_i \rangle \neq \langle K',\,-p_ie_i \rangle$. This fact immediately gives $\langle K,\,\alpha \rangle \neq \langle K',\,\alpha \rangle$ and hence $K|_{X_0-C_{p_i}}\neq {K'}|_{X_0-C_{p_i}}$. 
%Hence, $\langle L,\,p^2_ie_i \rangle \neq \langle L',\,p^2_ie_i \rangle$.
% Blow up formula shows that Corollary~\ref{cor:disjoint C_p}. Mayer-Vietoris exact sequence for $C_{p_i}\cup Y_0-C_{p_i}=Y_0$ implies that 
%Proposition~\ref{prop: widetilde{S}} implies that Seiberg-Witten basic classes of $\widetilde{S}(p_1,p_2,\dots,p_n)$ have mutually different restrictions to $\widetilde{S}(p_1,p_2,\dots,p_n)-\widetilde{D}_{p_i})$. 
\end{proof}
Though the rest of the proof is the same as that of Theorem~1.3 and 1.4 in \cite{AY3}, we proceed for the completeness. 
For a smooth $4$-manifold $Z$ we denote $N(Z)$ as the number of elements of $\beta(Z)$. 
\begin{lemma}\label{lem:N(X_i)}$N(X_i)=2^{p_i-1}N(X_0)$ $(1\leq i\leq n)$
\end{lemma}

\begin{proof}Corollary~\ref{cor:disjoint C_p}, Theorem~\ref{th:SW of E(n)} and Corollary~\ref{cor:SW4} guarantees that every Seiberg-Witten basic class of $X_0$ is a lift of some element of $\mathcal{C}(X_i')$. Lemma~\ref{lem:restriction} shows that these basic classes of $X_0$ have mutually different restrictions to $X_0-C_{p_i} (=X_i'-B_{p_i})$. Note that every element of $H^2(X_i';\mathbf{Z})$ is uniquely determined by its restriction to ${X_i'-B_{p_i}}$. (We can easily check this by using the cohomology exact sequence for the pair $(X_i', X_i'-B_{p_i})$.) Hence Theorems~\ref{th:SW1} and~\ref{th:SW2} give $N(X_i')=N(X_0)$. Now the required claim follows from the blow-up formula. 
\end{proof}

\begin{corollary}
If $p_1,p_2,\dots,p_n\geq 2$ are mutually different, then $X_i$ $(0\leq i\leq n)$ are mutually homeomorphic but not diffeomorphic.  
\end{corollary}

%\begin{proof}[Proofs of Theorem~$\ref{th:disjoint corks}$ and~$\ref{th:involutions of corks}$] 
\begin{proof}[Proof of Theorem~$\ref{th:disjoint corks}$ and Corollary~$\ref{cor:involutions of corks}$]
These clearly follow from the corollary above and Proposition~\ref{prop:X_i}.
\end{proof}
\begin{remark}If we appropriately choose $p_1,p_2,\dots,p_n$, then we can show that the natural combinations of cork twists of $X_0$ produce $2^{n}-1$ distinct smooth structures. In fact, similarly to Lemma~\ref{lem:N(X_i)}, we can show that the number of Seiberg-Witten basic classes of the combinations of cork twists of $X_0$ along $(W_{p_{i_1}-1}, f_{p_{i_1}-1})$, $(W_{p_{i_2}-1}, f_{p_{i_2}-1})$, $\cdots$, $(W_{p_{i_k}-1}, f_{p_{i_k}-1})$ is $2^{p_{i_1}+p_{i_2}+\dots+p_{i_k}-k}N(X_0)$. 
\end{remark}
%\begin{theorem}
%If $p_1,p_2,\dots,p_n\geq 2$ are mutually different, then the simply connected closed smooth $4$-manifold $X_0$ contains mutually disjoint copies of $W_{p_i}$ $(1\leq i\leq n)$ with the following properties:\smallskip \\
%$(1)$ Each $4$-manifold $X_{i}$ is obtained from $X_{0}$ by removing the copy of $W_{p_i}$ and regluing it via $f_{p_i}$. \smallskip \\
%$(2)$ The $4$-manifolds $X_{i}$ $(0\leq i\leq n)$ are mutually homeomorphic but not diffeomorphic. 
%\end{theorem}
%\begin{theorem}
%If $p_1,p_2,\dots,p_n\geq 2$ are mutually different, then the simply connected closed smooth $4$-manifold $X_0$ contains a copy of $W(p_1-1,p_2-1,\dots,p_n-1)$ with the following properties:\smallskip \\
%$(1)$ Each $4$-manifold $X_{i}$ is obtained from $X_{0}$ by removing the copy of $W(p_1-1,p_2-1,\dots,p_n-1)$ and regluing it via the involution $f^i(p_1-1,p_2-1,\dots,p_n-1)$. \smallskip \\
%$(2)$ The $4$-manifolds $X_{i}$ $(0\leq i\leq n)$ are mutually homeomorphic but not diffeomorphic. 
%\end{theorem}

%%%%%%%%%%%%%%%%%%%%%%%%%%%%%%%%%%%%%%%%%%%%%%%%%%%%%inyouowari

\section{Infinitely many disjointly embedded knotted cork}\label{sec:Infinitely many disjointly knotted cork}
Here we prove Theorem~\ref{th:infinite corks} and Corollary~\ref{cor:finite exotic in S^4}. In the previous sections, we used the embedding theorem of Stein $4$-manifolds to construct examples. However, in this section, we use the embedding theorem to detect smooth structures (i.e. to evaluate the minimal genera of surfaces which represent second homology classes).  For simplicity, we first give a proof for the $(W_1,f_1)$ cork. The argument is easily modified for any cork of Mazur type. 

\begin{definition}$(1)$ Let $M_n$ and $N_n$ $(n\geq 2)$ be the compact smooth $4$-manifolds with boundary given in Figure~\ref{fig12}. Note that $N_n$ is the cork twist of $M_n$ along $(W_1,f_1)$. 

$(2)$ Let $X_0$ be the simply connected non-compact smooth $4$-manifold given by attaching infinitely many handles to the boundary of $D^3\times [0,\infty)$ as shown in Figure~\ref{fig13}, where the $[0,\infty)$ component is horizontal line, and $\partial D^3=\mathbf{R}^2\cup \{\text{one point}\}$. In other words, $X_0$ is the infinite boundary sum of $\{M_i\mid i\geq 2\}$. Let $X_n$ be the cork twist of $X_0$ along $(W_1,f_1)$, where $W_1$ is the one contained in the $M_{n+1}$ component. 
\end{definition}
\begin{figure}[ht!]
\begin{center}
\includegraphics[width=3.4in]{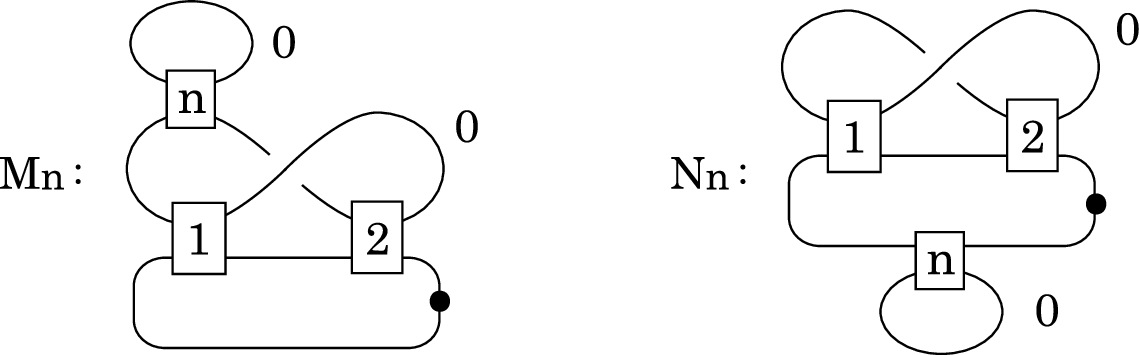}
\caption{}
\label{fig12}
\end{center}
\end{figure}
\begin{figure}[ht!]
\begin{center}
\includegraphics[width=4.4in]{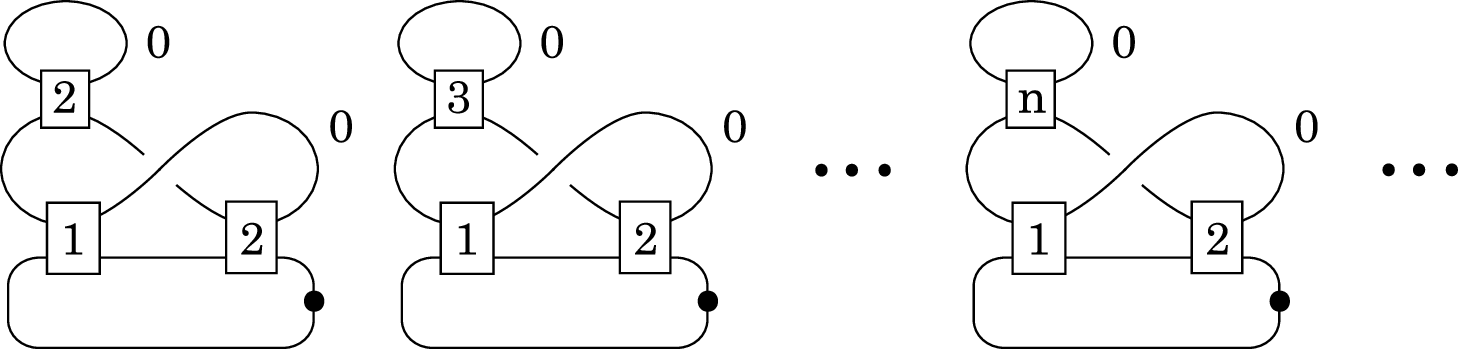}
\caption{$X_0$}
\label{fig13}
\end{center}
\end{figure}
\begin{lemma}\label{lem:genus and embedding of M_n}
$(1)$ The generator $\alpha$ of $H_2(N_n;\mathbf{Z})\cong \mathbf{Z}$ is represented by a smoothly embedded genus $n$ surface. 

$(2)$ For $n\geq 1$, the $4$-manifolds $M_n$ and $N_n$ can be embedded into the $4$-ball $D^4$. 
\end{lemma}
\begin{proof}$(1)$ The lower 2-handle (call $K$) of $N_n$ gives the generator $\alpha$ after sliding over the upper 2-handle $n$ times. Thus all we have to check is to see that the knot $K$ bounds a genus $n$ surface in the interior of $W_1$ after sliding $n$ times. See Figure~\ref{fig14}. Introduce a 1-handle/2-handle pair and slide the upper $0$-framed unknot twice, then we get the second picture. An isotopy gives the third picture. We slide the knot $K$ over the $0$-framed unknot $n$-times so that $K$ does not link with the lower dotted circle. We get the fourth picture by ignoring two $2$-handles and isotopy. We can now easily see that $K$ bounds a genus $n$ surface by the standard argument (cf. Gompf-Stipsicz~\cite[Exercise.4.5.12.(b)]{GS}). (Check that $K$ is the boundary of $D^2$ with $2n$ bands attached.)

\begin{figure}[ht!]
\begin{center}
\includegraphics[width=3.9in]{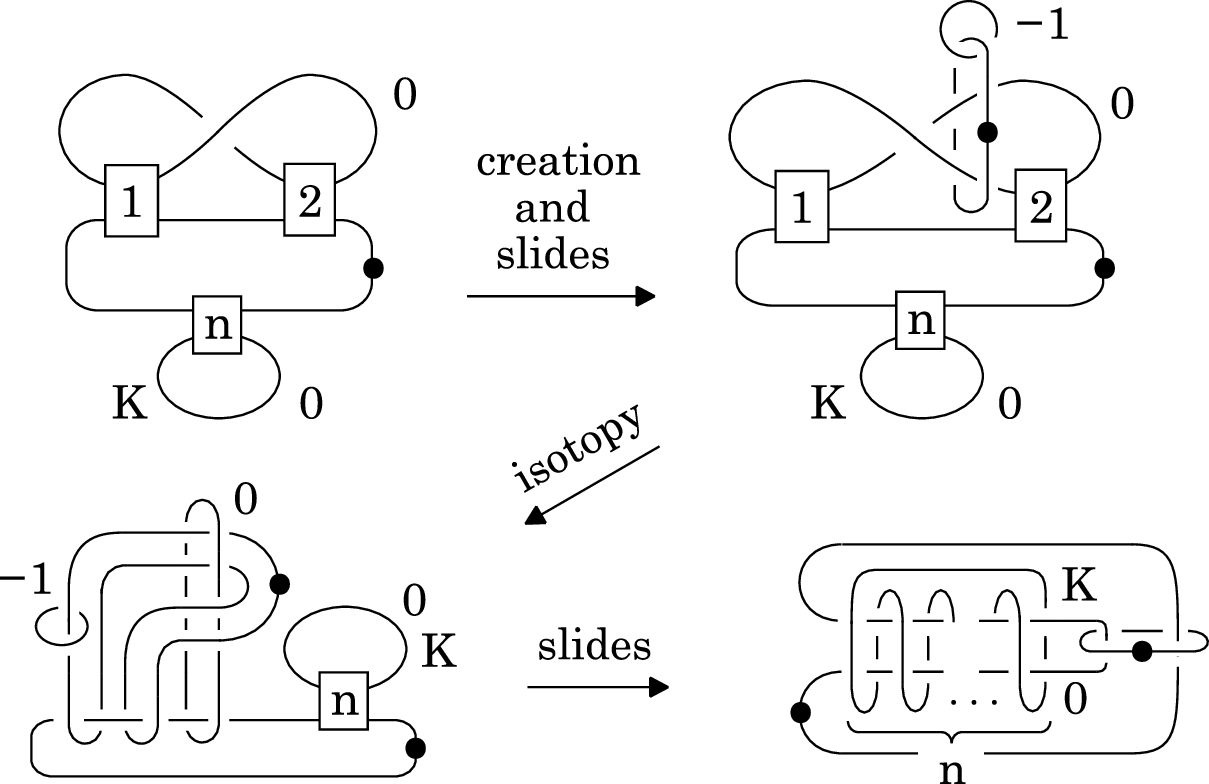}
\caption{}
\label{fig14}
\end{center}
\end{figure}

$(2)$ Attach a $2$-handle to the boundary of $M_n$ as in the first picture of Figure~\ref{fig15}. The second picture is given by sliding the upper left 2-handle over this new 2-handle $n$-times. Slide the middle 2-handle over its meridian as in the third diagram. Note that the middle 2-handle now links with the 1-handle geometrically once. Cancelling the 1-handle gives the last diagram. Attach two $3$-handles cancelling these two $2$-handles. So get an embedding of $M_n$ into $D^4$. The $N_n$ case is similar, Figure~\ref{fig16}.\end{proof}
\begin{figure}[ht!]
\begin{center}
\includegraphics[width=4.1in]{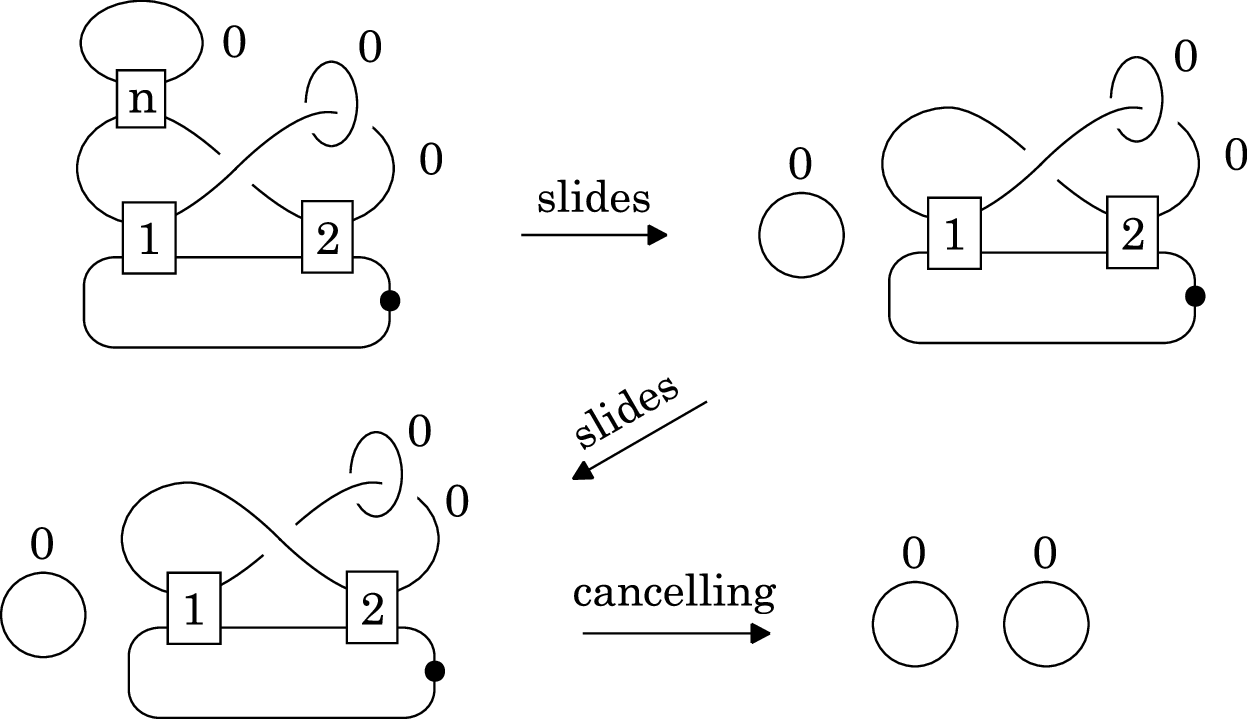}
\caption{}
\label{fig15}
\end{center}
\end{figure}\newpage
\begin{figure}[ht!]
\begin{center}
\includegraphics[width=3.1in]{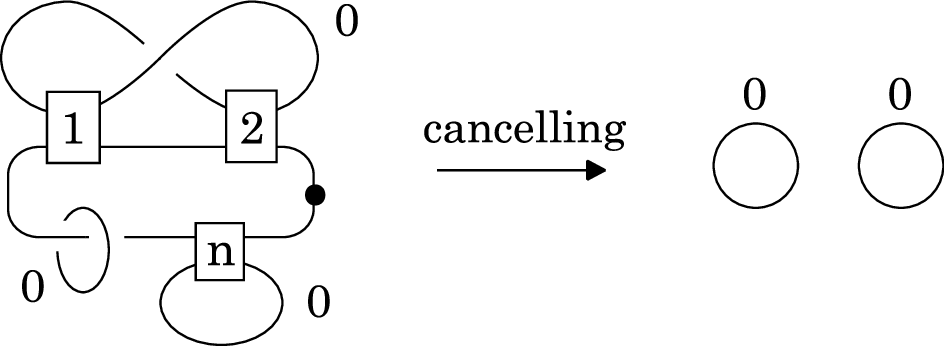}
\caption{}
\label{fig16}
\end{center}
\end{figure}

%\newpage

By using the embedding theorem of Stein $4$-manifolds, we show the following. 
\begin{lemma}\label{lem:upper bound of genus of N_n}
Let $\alpha$ be the generator of $H_2(N_n;\mathbf{Z})\cong \mathbf{Z}$. If $k\alpha$ is represented by a smoothly embedded surface with genus less than $n$, then $k=0$. 
\end{lemma}
\begin{proof}
Let $\widetilde{N}_n$ be the compact smooth $4$-manifold given in the left side of Figure~\ref{fig17}. The right side is a Legendrian diagram of $\widetilde{N}_n$ with contact $-1$-framings. Thus $\widetilde{N}_n$ is a compact Stein $4$-manifold. Therefore there exists a simply connected minimal closed symplectic $4$-manifold $S_n$ which contains $\widetilde{N}_n$. 

Let $e_1,e_2,\dots,e_{n-1}$ be the standard orthogonal basis of $H_2((n-1)\overline{\mathbf{C}\mathbf{P}^2};\mathbf{Z})$ such that $e^2_i=-1$ $(1\leq i\leq n-1)$. Blow up $S_n$ as in Figure~\ref{fig18}, where $e_1,e_2,\dots,e_{n-1}$ denote the second homology classes given by corresponding $2$-handles. Note that this picture contains $N_n$. The lower $0$-framed unknot in this picture gives the generator $\alpha$ of $H_2(N_n;\mathbf{Z})$ after sliding over the upper $0$-framed unknot $n$ times. We thus get $\alpha\cdot e_1=n$ and $\alpha\cdot e_i=1$ $(2\leq i\leq n-1)$, where we view $\alpha$ as the element of $H_2(S_n\#(n-1)\overline{\mathbf{C}\mathbf{P}^2})$ through the natural inclusion.

Let $K$ be a Seiberg-Witten basic class of $S_n$. Then the blow up formula shows that $\pm K+ e_1+ e_2+ \dots+e_{n-1}$ is a Seiberg-Witten basic class of $S_n\#(n-1)\overline{\mathbf{C}\mathbf{P}^2}$. Therefore, there exists a Seiberg-Witten basic class $L$ of $S_n\#(n-1)\overline{\mathbf{C}\mathbf{P}^2}$ which satisfies the inequality below:
\begin{equation*}
\lvert \langle L, k\alpha \rangle\rvert \geq \lvert k (n+(n-2))\rvert =\lvert k \rvert(2n-2). 
\end{equation*}
Since $(k\alpha)^2=0$ and $n\geq 2$, we can now easily check the required claim by applying the adjunction inequality to $S_n\#(n-1)\overline{\mathbf{C}\mathbf{P}^2}$. 
\end{proof}
\begin{figure}[ht!]
\begin{center}
\includegraphics[width=3.5in]{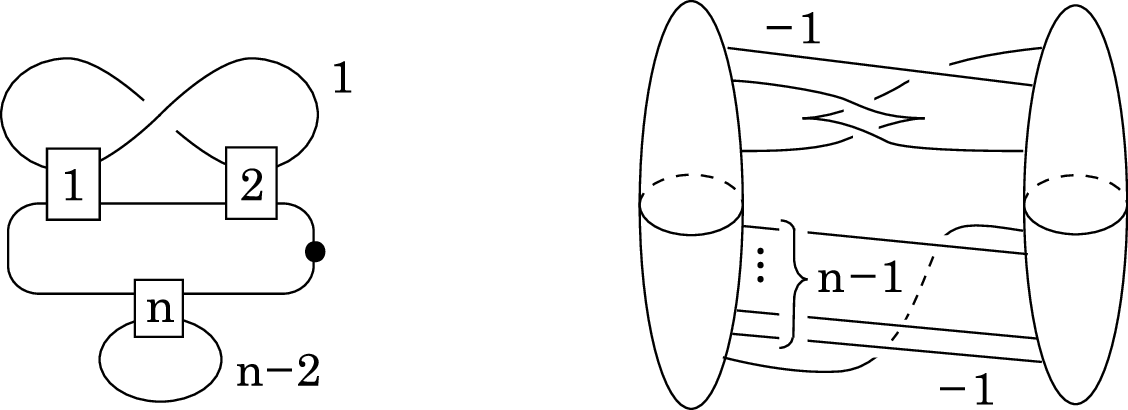}
\caption{$\widetilde{N}_n$ and its legendrian diagram with contact framings}
\label{fig17}
\end{center}
\end{figure}
\begin{figure}[ht!]
\begin{center}
\includegraphics[width=1.15in]{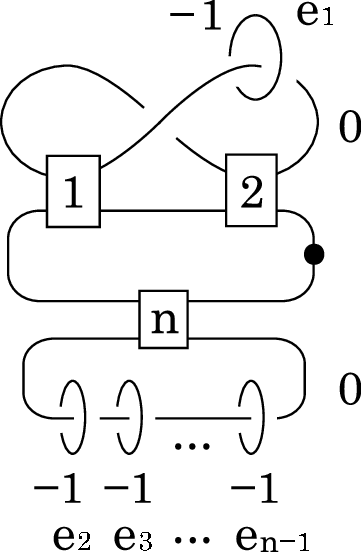}
\caption{The subhandlebody of $S_n\#(n-1)\overline{\mathbf{C}\mathbf{P}^2}$}
\label{fig18}
\end{center}
\end{figure}
\newpage
\begin{proposition}\label{prop:basis of X_n}
$(1)$ For each $n\geq 1$, there exists a basis $\{\alpha, \beta_1,\beta_2,\dots\}$ of $H_2(X_n;\mathbf{Z})$ such that $\alpha$ is represented by a smoothly embedded surface with genus $n+1$, and that $\beta_i$ $(i\geq 1)$ is represented by a smoothly embedded sphere. 

$(2)$ For any $n\geq 1$, there exists no basis $\{\alpha, \beta_1,\beta_2,\dots\}$ of $H_2(X_n;\mathbf{Z})$ such that $\alpha$ is represented by a surface with genus less than $n+1$, and that $\beta_i$ $(i\geq 1)$ is represented by a smoothly embedded sphere. 
\end{proposition}
\begin{proof}$(1)$ is obvious from~Lemma~\ref{lem:genus and embedding of M_n}. \smallskip 

$(2)$. Let $u_i$ $(i\geq  2,\; i\neq n+1)$ and $u_{n+1}$ be the basis of $H_2(M_i;\mathbf{Z})$ and $H_2(N_{n+1};\mathbf{Z})$, respectively. Suppose that an element $v=\sum_{i=2}^{\infty} a_iu_i$ of $H_2(X_n;\mathbf{Z})$ is represented by a surface with genus less than $n+1$. Here $a_i=0$ except finite number of $i$. Lemma~\ref{lem:genus and embedding of M_n} shows that $X_n$ is embedded into $N_{n+1}$ and that this embedding sends $v$ to $a_{n+1}u_{n+1}$. Since $a_{n+1}u_{n+1}\in H_2(N_{n+1};\mathbf{Z})$ is represented by a surface with genus less than $n+1$, Lemma~\ref{lem:upper bound of genus of N_n} gives $a_{n+1}=0$. This fact implies the required claim. 
%there exists such a basis $\{\alpha, \beta_1,\beta_2,\dots\}$ of $H_2(X_n;\mathbf{Z})$. Let $u_i$ $(i\geq  2,\; i\neq n)$ and $u_n$ be the basis of $H_2(M_i;\mathbf{Z})$ and $H_2(N_n;\mathbf{Z})$, respectively. Then $\alpha=\sum_{i=2}^{\infty} a_iu_i$ where $a_i=0$ except finite number of $i$. Lemma~\ref{lem:genus and embedding of M_n} shows that $X_n$ is embedded into $N_n$ and that this embedding sends $\alpha$ to $a_nu_n$. 
%Then $\{u_i\mid i\geq 2\}$ is a basis of $H_2(X_n;\mathbf{Z})$. $\alpah$
\end{proof}
\begin{proof}[Proof of Theorem~\ref{th:infinite corks}]We first discuss for the $(W_1, f_1)$ cork. 
The definition of $X_n$ obviously shows $(1), (2)$.  The claim $(3)$ follows from Proposition~\ref{prop:basis of X_n}. \smallskip

Now consider a general cork $(C,\tau)$ of Mazur type. We modify the above construction as follows. Let $N_n$ be the 4-manifold obtained from $C$ by attaching a 2-handle similarly to Figure~\ref{fig12}. Then define $M_n$ as the cork twist of $N_n$ along $(C,\tau)$. Using these $M_n$ and $N_n$, we can define $X_n$ similarly to the original definition. The claims corresponding to Lemmas~\ref{lem:genus and embedding of M_n} and \ref{lem:upper bound of genus of N_n} clearly hold after suitable modifications of values of various genera (Though we calculated the exact values of genera in those claims, we do not need the exact values if we care about detecting smooth structures of infinitely many of $X_n$'s. ). Namely we obtain the following (For the proof of Lemma~\ref{final:lem:embedding}.(2), see also Lemma~\ref{cork:lem:S4}.).
\begin{lemma}\label{final:lem:embedding}
$(1)$ There exist an integer sequence $g_n$ $(n\geq 1)$ satisfying the following condition. The generator $\alpha$ of $H_2(N_n;\mathbf{Z})\cong \mathbf{Z}$ is represented by a smoothly embedded genus $g_n$ surface. We may assume that $g_{n+1}>g_n$ for each $n\geq 1$, if necessary by connect summing with null homologous surfaces. 

$(2)$ For $n\geq 1$, the $4$-manifolds $M_n$ and $N_n$ can be embedded into the $4$-ball $D^4$. 
\end{lemma}
\begin{lemma}
Let $\alpha$ be the generator of $H_2(N_n;\mathbf{Z})\cong \mathbf{Z}$. Then there exists a strictly increasing integer sequence $h_n$ $(n\geq 1)$ satisfying the following condition. If $k\alpha$ is represented by a smoothly embedded surface with genus less than $h_n$, then $k=0$. 
\end{lemma}
\begin{proposition}\label{prop:final}
$(1)$ For each $n\geq 1$, there exists a basis $\{\alpha, \beta_1,\beta_2,\dots\}$ of $H_2(X_n;\mathbf{Z})$ such that $\alpha$ is represented by a smoothly embedded surface with genus $g_{n+1}$, and that $\beta_i$ $(i\geq 1)$ is represented by a smoothly embedded sphere. 

$(2)$  For any $n\geq 1$, there exists no basis $\{\alpha, \beta_1,\beta_2,\dots\}$ of $H_2(X_n;\mathbf{Z})$ such that $\alpha$ is represented by a surface with genus less than $h_{n+1}$, and that $\beta_i$ $(i\geq 1)$ is represented by a smoothly embedded sphere. 
\end{proposition}
Since $g_n$ and $h_n$ $(n\geq 1)$ are both strictly increasing integer sequence, Proposition~\ref{prop:final} implies the existence of a strictly increasing integer sequence $n_i$ $(i\geq 1)$ such that $X_{n_i}$ $(i\geq 1)$ are mutually non-diffeomorphic. The rest of the required claims easily follows. 
\end{proof}
\begin{remark}Lemma~\ref{lem:genus and embedding of M_n}.(2) shows that $X_n$ $(n\geq 0)$ can be embedded into $S^4$. 
\end{remark}
\begin{definition}Let $X^{(n)}_0$ $(n\geq 1)$ be the boundary sum of $M_2,M_3,\dots,M_{n+1}$. Define $X^{(n)}_i$ $(1\leq i\leq n)$ as the cork twist of $X^{(n)}_0$ along $(W_1,f_1)$ where this $W_1$ is the one contained in the $M_{i+1}$ component.
\end{definition}
\begin{proof}[Proof of Corollary~\ref{cor:finite exotic in S^4}]This is almost the same as the proof of Theorem~\ref{th:infinite corks}. 
\end{proof}
%%%%%%%%%%%%%%%%%%%%%%%%%%%%%%%%%%%%%%%%%%%%%%%%%%%%%%%%%%%%%%
%\section{Finitely many exotic Stein 4-manifolds}
%In this section, we construct the following example. 
%\begin{theorem}
%Fon each second Betti number $n\geq 1$, there exist simply connected compact Stein $4$-manifolds $S_i$ $(1\leq i\leq 2^n)$ with $b_2=n$ which are mutually homeomorphic but not diffeomorphic. 
%\end{theorem}
%\begin{definition}
%Let $W_{k_1,k_2,\dots,k_n}$
%\end{definition}
%%%%%%%%%%%%%%%%%%%%%%%%%%%%%%%%%%%%%%%%%%%%%%%%%%%%%%%%%%%%%
\section{Further remarks}\label{section:remark}
In this section, we conclude this paper by making some remarks. 

%\subsection{Plugs}\label{subsection:plug}
In~\cite{AY1}, we introduced new objects which we call plugs. By using the embedding theorem of Stein $4$-manifolds, we can easily construct examples of plug structures corresponding to Theorem~\ref{th:knotting corks}, \ref{th:disjoint corks} and Corollary~\ref{cor:involutions of corks}. The proofs are almost the same as  that of cork structures. See also~\cite{AY3}. 

In Section~\ref{sec:Infinitely many disjointly knotted cork}, we used the embedding theorem to detect smooth structures. This technique is useful for constructions of exotic smooth structures on compact 4-manifolds with boundary. For details, see our paper \cite{AY5} which was written after the first draft of this paper. 


\begin{thebibliography}{20}
\bibitem{A1}S.\ Akbulut, \textit{A fake compact contractible $4$-manifold}, J.\ Differential Geom.\ \textbf{33} (1991), no.\ 2, 335--356.
%\bibitem{A6}S.\ Akbulut, \textit{A solution to a conjecture of Zeeman}, Topology, vol.30, no.3, (1991), 513--515. 
%\bibitem{A2}S.\ Akbulut, \textit{An exotic $4$-manifold}, J.\ Differential Geom.\ \textbf{33} (1991), no.\ 2, 357--361.
%\bibitem{A3}S.\ Akbulut, \textit{Constructing a fake $4$-manifold by Gluck construction  to a standard $4$-manifold,}, Topology, vol.27, no.\ 2 (1988), 239-243.

%\bibitem{A4}S.\ Akbulut, \textit{A fake cusp and a fishtail}, Proceedings of 6th Gokova Geometry-Topology Conference, Turkish J. Math. \textbf{23} (1999), no.\ 1, 19--31. 
\bibitem{A4}S.\ Akbulut, \textit{Variations on Fintushel-Stern knot surgery on 4-manifolds}, Turkish J. Math. \textbf{26} (2002), no. 1, 81--92.
\bibitem{A5}S.\ Akbulut, \textit{The Dolgachev Surface}, arXiv:0805.1524. 
%\bibitem{AD}S.\ Akbulut and S.\ Durusoy, \textit{An involution acting nontrivially on Heegaard-Floer homology}, Geometry and topology of manifolds, 1--9, Fields Inst. Commun., 47, Amer. Math. Soc., Providence, RI, 2005.
\bibitem{AKa}S.\ Akbulut and C.\ Karakurt, \textit{Action of the cork twist on Floer homology}, arXiv:1104.2247.
\bibitem{AM1}S.\ Akbulut and R.\ Matveyev, \textit{Exotic structures and adjunction inequality}, Turkish J. Math. \textbf{21} (1997), no.\ 1, 47--53.
\bibitem{AM2}S.\ Akbulut and R.\ Matveyev, \textit{A convex decomposition theorem for $4$-manifolds}, Internat. Math. Res. Notices 1998, no.\ 7, 371--381.
%\bibitem{AO1}S.\ Akbulut and B.\ Ozbagci, \textit{Lefschetz fibrations on compact Stein surfaces}, Geom.\ Topol.\ 5 (2001), 319--334.
%\bibitem{AO2}S.\ Akbulut and B.\ Ozbagci, \textit{Erratum: ``Lefschetz fibrations on compact Stein surfaces'' \textnormal{[Geom.\ Topol.\ 5 (2001), 319--334.]}}, Geom.\ Topol.\ 5 (2001), 939--945
\bibitem{AO3}S.\ Akbulut and B.\ Ozbagci, \textit{On the topology of compact Stein surfaces}, Int.\ Math.\ Res.\ Not.\ 2002, no.\ 15, 769--782.
\bibitem{AY1}S.\ Akbulut and K.\ Yasui, \textit{Corks, Plugs and exotic structures}, Journal of G\"{o}kova Geometry Topology, volume \textbf{2} (2008), 40--82.  
%\bibitem{AY2}S.\ Akbulut and K.\ Yasui, \textit{Small exotic Stein manifolds}, arXiv:0807.3815, to apperar in Commentarii Mathematici Helvetici. 
\bibitem{AY3}S.\ Akbulut and K.\ Yasui, \textit{Knotting corks}, Journal of Topology \textbf{2} (2009), 823--839. 
\bibitem{AY5}S.\ Akbulut and K.\ Yasui, \textit{Cork twisting exotic Stein 4-manifolds}, arXiv:1102.3049, to appear in J.\ Differential Geom.
%\bibitem{AEMS}A.\ Akhmedov, J.\ B.\ Etnyre, T.\ E.\ Mark, and I. Smith, \textit{A note on Stein fillings of contact manifolds}, Math. Res. Lett. \textbf{15} (2008), no. 6, 1127--1132. 
\bibitem{Au}D.\ Auckly, \textit{Families of four-dimensional manifolds that become mutually diffeomorphic after one stabilization}, Topology Appl. \textbf{127} (2003), no. 1-2, 277--298. 
\bibitem{BG}\v Z.\ Bi\v zaca and R.\ E.\ Gompf, \textit{Elliptic surfaces and some simple exotic $R\sp 4$'s}, J. Differential Geom. \textbf{43} (1996), no. 3, 458--504. 
%\bibitem{B}S.\ Boyer, \textit{Simply-connected $4$-manifolds with a given boundary}, Trans. Amer. Math. Soc. \textbf{298} (1986), no.\ 1, 331--357.
\bibitem{C}C.\ L.\ Curtis, M.\ H. Freedman, W.\ C.\ Hsiang, and R.\ Stong, \textit{A decomposition theorem for $h$-cobordant smooth simply-connected compact $4$-manifolds}, Invent. Math. \textbf{123} (1996), no.\ 2, 343--348.
\bibitem{E1}Y.\ Eliashberg. \textit {Topological characterization of Stein manifolds of dimension $>2$}, International J. of Math. Vol. 1 (1990), No 1  pp. 29-46.

%\bibitem{E2}Y.\ Eliashberg. \textit{Filling by holomorphic discs and its applications}, Geometry of low-dimensional manifolds, 2 (Durham, 1989), 45--67, London Math. Soc. Lecture Note Ser., 151, Cambridge Univ. Press, Cambridge, 1990. 
\bibitem{FS1}R.\ Fintushel and R.\,J.\ Stern, 
\textit{Rational blowdowns of smooth $4$--manifolds}, 
J.\ Differential Geom.\ \textbf{46} (1997), no.\ 2, 181--235.
\bibitem{FS2}R.\ Fintushel and R.\,J.\ Stern, \textit{Knots, links, and $4$-manifolds}, Invent. Math. \textbf{134} (1998), no. 2, 363--400.
\bibitem{FS3}R.\ Fintushel and R.\,J.\ Stern, \textit{Six Lectures on Four 4-Manifolds}, arXiv:math/0610700.
%\bibitem{G}R.\,E.\ Gompf, \textit{Nuclei of elliptic surfaces}, Topology \textbf{30} (1991), no.\ 3, 479--511.
\bibitem{GS}R.\,E.\ Gompf and A.\,I.\ Stipsicz, 
\textit{$4$-manifolds and Kirby calculus}, Graduate Studies in Mathematics, \textbf{20}. American Mathematical Society, 1999.
%\bibitem{HKK}J.\ Harer, A.\ Kas and R.\ Kirby, \textit{Handlebody decompositions of complex surfaces}, Mem.\ Amer.\ Math.\ Soc.\ \textbf{62} (1986), no.\ 350.
%\bibitem{K}R.\ Kirby, \textit{Akbulut's corks and $h$-cobordisms of smooth, simply connected $4$-manifolds}, Turkish J. Math. \textbf{20} (1996), no. 1, 85--93.\bibitem{LM}P.\ Lisca and G.\ Mati\'c, \textit{Tight contact structures and Seiberg-Witten invariants}, Invent. Math. \textbf{129} (1997), no. 3, 509--525.
%\bibitem{KM1}P.\ Kronheimer and T.\ Mrowka, \textit{The genus of embedded surfaces in thevprojective plane}, Math.\ Res.\ Lett.\ \textbf{1} (1994), 797--808.
%\bibitem{Lic}W.\ B.\ R.\ Lickorish, \textit{Knotted contractible 4-manifolds in $S\sp 4$}, Pacific J.\ Math.\ \textbf{208} (2003), no.\ 2, 283--290.
\bibitem{LM1}P.\ Lisca and G.\ Mati\'{c}, \textit{Tight contact structures and Seiberg-Witten invariants}, Invent.\ Math.\ \textbf{129} (1997) 509--525.
%\bibitem{LM2}P.\ Lisca and G.\ Mati\'{c}, \textit{Stein 4-manifolds with boundary and contact structures}, Topology and its Applications \textbf{88} (1998) 55--66.
%\bibitem{Liv}C.\ Livingston, \textit{Observations on Lickorish knotting of contractible 4-manifolds}, Pacific J.\ Math.\ \textbf{209} (2003), no.\ 2, 319--323.
\bibitem{M}R.\ Matveyev, \textit{A decomposition of smooth simply-connected $h$-cobordant $4$-manifolds}, J.\ Differential Geom.\ \textbf{44} (1996), no.\ 3, 571--582.
%\bibitem{Ma}B.\ Mazur, \textit{A note on some contractible $4$-manifolds}, 
%Ann. of Math.\ \textbf{73} (1961), 221--228.
%\bibitem{MST}J.\ Morgan, Z.\ Szab\'{o} and C.\ Taubes, \textit{A product formula for the Seiberg-Witten invariants and the generalized Thom conjecture}, J.\ Differential Geom.\ \textbf{44} (1996), 706--788.
\bibitem{OS1}B.\ Ozbagci and A I.\ Stipsicz, \textit{Surgery on contact 3-manifolds and Stein surfaces}, Bolyai Society Mathematical Studies, 13. Springer-Verlag, Berlin; Janos Bolyai Mathematical Society, Budapest, 2004.
%\bibitem{OS2}B.\ Ozbagci and A I.\ Stipsicz, \textit{Contact 3-manifolds with infinitely many Stein fillings}, Proc. Amer. Math. Soc. \textbf{132} (2004), no. 5, 1549--1558.
\bibitem{OzSz}P.\ Ozsv\'{a}th and Z.\ Szab\'{o}, \textit{The symplectic Thom conjecture}, Ann.\ of Math.\ \textbf{151} (2000), 93--124. 
%\bibitem{OS}P.\ Ozsv\'{a}th and Z.\ Szab\'{o}, \textit{The Dehn surgery characterization of the trefoil and the figure eight knot}, arXiv:math/0604079. 
\bibitem{P1}J.\ Park, \textit{Seiberg-Witten invariants of generalised rational blow-downs}, 
Bull.\ Austral.\ Math.\ Soc.\ \textbf{56} (1997), no.\ 3, 363--384.
%\bibitem{Sz}Z.\ Szab\'{o}, \textit{Exotic $4$-manifolds with $b\sp +\sb 2=1$}, Math. Res. Lett. \textbf{3} (1996), no. 6, 731--741.
%\bibitem{Y0}K.\ Yasui, \textit{An exotic rational elliptic surface without $1$- or $3$-handles}, in Intelligence of Low Dimensional Topology 2006 (J.\ Carter \textit{et al.} ed.), Series on Knots and Everything, vol. 40, World Scientific Publishing Co. 2007, 375--382.
\bibitem{W}E.\ Witten, \textit{Monopoles and four-manifolds}, Math.\ Res.\ Lett.\ \textbf{1} (1994), no.\ 6, 769--796. 
\bibitem{Y1}K.\ Yasui, \textit{Exotic rational elliptic surfaces without $1$-handles}, Algebr.\ Geom.\ Topol.\ \textbf{8} (2008), 971--996.
%\bibitem{Y2}K.\ Yasui, \textit{Small exotic rational surfaces without $1$- and $3$-handles}, arXiv:0807.0373.
%\bibitem{Y3}K.\ Yasui, \textit{Homeomorphisms between rational blow-downs preserving the Seiberg-Wittten invariants}, in preparation.
%\bibitem{Y4}K.\ Yasui, \textit{Elliptic surfaces without $1$-handles}, arXiv:0802.3372, to appear in Journal of Topology. 
\end{thebibliography}
\end{document}